\documentclass{article}

\usepackage[utf8]{inputenc} 
\usepackage[T1]{fontenc} 

\usepackage{amsfonts} 
\usepackage{amsmath}
\usepackage{amsmath}
\usepackage{amssymb}
\usepackage{amsthm}

\usepackage{hyperref} 
\usepackage{url} 

\usepackage{color,soul}
\usepackage{xcolor} 
\usepackage[parfill]{parskip}
\usepackage{geometry}
\usepackage{algorithm}
\usepackage[noend]{algpseudocode}
\usepackage{bbm}
\usepackage{comment}
\usepackage{enumerate}
\usepackage{enumitem}
\usepackage{float}
\usepackage{graphicx}
\usepackage{nicefrac} 

\newtheorem{theorem}{Theorem}

\newtheorem{lemma}[theorem]{Lemma}
\newtheorem{corollary}[theorem]{Corollary}

\newtheorem{innercustomthm}{Assumption}
\newenvironment{customhyp}[1]
 {\renewcommand\theinnercustomthm{#1}\innercustomthm}
 {\endinnercustomthm}
\theoremstyle{definition}

\newcommand{\Ebb}{\mathbb{E}}
\newcommand{\Nbb}{\mathbb{N}}
\newcommand{\Pbb}{\mathbb{P}}

\newcommand{\Rbb}{\mathbb{R}}

\newcommand{\Zbb}{\mathbb{Z}}

\newcommand{\Kbf}{\mathbf{K}}
\newcommand{\Lbf}{\mathbf{L}}
\newcommand{\Vbf}{\mathbf{V}}

\newcommand{\Fcal}{\mathcal{F}}

\newcommand{\Hcal}{\mathcal{H}}
\newcommand{\Mcal}{\mathcal{M}}
\newcommand{\Ocal}{\mathcal{O}}

\newcommand{\Qcal}{\mathcal{Q}}
\newcommand{\Xcal}{\mathcal{X}}

\newcommand{\dommes}{\mu}
\newcommand{\one}{{\mathbf{1}}}

\renewcommand{\geq}{\geqslant}
\renewcommand{\leq}{\leqslant}

\DeclareMathOperator{\argmin}{\text{arg min}}
\DeclareMathOperator{\argmax}{\text{arg max}}
\DeclareMathOperator{\crit}{crit}
\DeclareMathOperator{\DKL}{\text{KL}}

\DeclareMathOperator{\pen}{\text{pen}}

\makeatletter
\newcommand*\bigcdot{\mathpalette\bigcdot@{.5}}
\newcommand*\bigcdot@[2]{\mathbin{\vcenter{\hbox{\scalebox{#2}{$\m@th#1\bullet$}}}}}
\makeatother

\newcommand{\luc}[1]{\textcolor{blue}{[#1]}}

\title{General oracle inequalities for a penalized log-likelihood criterion based on non-stationary data}
\date{}

\author{
 Julien Aubert*, Luc Lehéricy* and Patricia Reynaud-Bouret* \\
 {\small *Université Côte d'Azur, CNRS, LJAD, France}
}

\begin{document}

\maketitle

\begin{abstract}
We prove oracle inequalities for a penalized log-likelihood criterion that hold even if the data are not independent and not stationary, based on a martingale approach. The assumptions are checked for various contexts: density estimation with independent and identically distributed (i.i.d) data, 
hidden Markov models, spiking neural networks, adversarial bandits. In each case, we compare our results to the literature, showing that, although we lose some logarithmic factors in the most classical case (i.i.d.), these results are comparable or more general than the existing results in the more dependent cases.

\end{abstract}



\section{Introduction}

Maximum likelihood estimator (MLE) (see \cite{bickel-doksum} and the references therein) is often considered a graal, even if it has been often debated w.r.t. robustness in particular \cite{baraud-birge}. In particular, in the i.i.d. setting, it is known to be asymptotically optimal (variance which is asymptotically minimal w.r.t. the Cramer-Rao bound) under mild conditions (typically differentiability of the density) \cite{Lehmann-book}. In 1973, \cite{akaike} proposed its famous penalized log-likelihood criterion (AIC) stating that to select in a finite set of models, one should penalize the log-likelihood of the model $m$ by $D_m$, the number of parameters describing $m$. There is a large variety of variants of AIC, for which the asymptotic properties are more or less precise (see \cite{Burnham} and the references therein).
Under additional differentiability assumptions, the Wilks phenomenon \cite{wilks} quantifies more precisely (in an asymptotic way) how the recentered maximal log-likelihood behaves as a chi-square distribution with $D_m$ degrees of freedom. This phenomenon allows to construct asymptotic likelihood ratio test and therefore to perform model selection by multiple testing \cite{zheng}. This idea has been used in many settings, even quite recently in combination with asymptotic $\ell_1$ model selection \cite{tang, Candes2019, Candes2019-court}. In short, with AIC-like penalties and the Wilks phenomenon, we have a quite clear asymptotic picture in the i.i.d. setting on how the maximum likelihood behaves in terms of dimension and how it has to be penalized to find the correct model in a finite fixed set of smooth enough models.

Moving on to the non-asymptotic setting, things are more difficult to study. It is possible to get concentration inequalities that mimic the Wilks phenomenon \cite{massart-wilk}. With this kind of non asymptotic results, it is possible to understand non asymptotically what the MLE does in an exponential family \cite{spok2012}, or how to penalize it to make it work even in infinite dimensional settings with finite effective dimension \cite{spok2017}. It is also possible to penalize the log-likelihood to perform model selection \cite{castellan, massart2007concentration} for particular family of models in the i.i.d. setting. However, in these works, model selection "à la Birgé-Massart" \cite{birge-massart} is based on a form of linearity between the contrast and the family of models under consideration, which leads to a very precise tuning of the penalty constants \cite{arlot} but which at the same time, prevents the results to be applied in more general settings.

If we drop the independence assumption but keep the stationarity of the data, there have been a vast variety of model selection results by minimization of $\ell_0$ or $\ell_1$ penalized contrasts with or without log-likelihood. Let us cite but a few: Markov chains \cite{lacourMarkov}, hidden Markov models \cite{lacourHMM, dCGL2016minimax, lehericy2021hmm_mle}, spiking neuronal networks with unobserved components \cite{OstRB}, point processes \cite{Hansen_2015}. Each time, the arguments are a combination of martingale approach and non asymptotic exponential inequalities, that derive from the ergodicity of the process and mixing properties.

However, the Akaike criterion (minus log-likelihood plus penalization proportional to $D_m$) is routinely used even beyond this setting \cite{daw2011trial, wilson2019ten}. In particular, in learning experiments where the individual has to learn how to perform a task, and learn it only once, data are neither independent nor stationary. Some authors \cite{ramponi2020inverse} have tried to assume that individuals are identically distributed, but this is a very strong and quite unlikely setting in practice. In a first work \cite{aubert23}, we proved meaningful bounds for the maximum likelihood estimator on an individual learning trajectory. However, up to our knowledge, there is no theoretical work on model selection in this setting or more generally for dependent and non stationary data.

The purpose of this work is to derive a non-asymptotic oracle inequality for AIC-like model selection 
that is general enough to cover all the setups listed before: i.i.d. samples, Markov chains, hidden Markov models, partially observed neuronal networks, learning models and more.

The proof relies on a martingale approach and an exponential inequality that was recently proved in \cite{aubert:hal-04526484} for a supremum of empirical centered processes that are stochastically normalized. This exponential inequality generalizes the works of \cite{baraud}, \cite{talagrand}, and is inspired by the works on renormalized martingales due to \cite{bercu2015concentration, bercu2019new} and \cite{de_la_Pe_a_2004}. This is the key non asymptotic exponential inequality in the present work.

The rest of the paper is organised as follows. In Section~\ref{sec:frame}, we introduce the notations and the general framework. In Section~\ref{sec:main}, we state our assumptions and oracle inequalities in probability and in expectation, in the bounded and unbounded frameworks. We also discuss precisely how the penalty in $D_m$ is obtained with respect to more classical proof techniques. In Section~\ref{sec:appli}, we discuss various cases that are covered by these general results in light of the existing results of the literature. The appendices are dedicated to the proofs.

\section{Framework and notation \label{sec:frame}}

Given two integers $a \leq b$ and a sequence $(x_s)_{s \in \Zbb}$, write $x_a^b = (x_a, \dots, x_b)$ (with $x_a^b$ being the empty sequence when $a > b$). For any two real numbers $x,y$, write $x \vee y$ their maximum and $x \wedge y$ their minimum. Let $\Nbb^*$ be the set of positive integers, and for any $n \in \Nbb^*$, write $[n] = \{1, \dots, n\}$. Finally, $\log$ denotes the natural logarithm.

Let $n \geq 2$. We observe a process $(X_t)_{1 \leq t \leq n}$ defined on a polish measure space $(\Xcal, \Fcal, \dommes)$ and adapted to a filtration $(\Fcal_t)_{1 \leq t \leq n}$. We write $\Pbb$ the corresponding probability and $\Ebb$ the corresponding expectation.

We are interested in the successive conditional distributions of $X_t$.
If $\Xcal$ is discrete and $\mu$ is the counting measure, we are therefore interested by the sequence
\begin{equation*}
 p^\star_t(.)=\Pbb(X_t=.|\Fcal_{t-1}), \quad \forall t\in [n].
\end{equation*}
More generally, in the sequel, for general measured spaces $(\Xcal, \Fcal, \mu)$, we always assume that the conditional density of $X_t$ given $\Fcal_{t-1}$ with respect to $\mu$ exists and we write it $p^\star_t(.)$. Note that for all $x \in \Xcal, (p^\star_t(x))_{1 \leq t \leq n}$ is therefore predictable with respect to the filtration--we write that $p^\star_t$ is predictable for short. We write $p^\star=(p^\star_t)_{t\in[n]}$ the vectors of all the successive conditional densities.

\subsection{Some examples}

The emblematic case is the case where we take for filtration $\Fcal_t=\sigma(X_1^t)$ for $t\geq 1$ and $\Fcal_0$ is the trivial $\sigma$-algebra. In this case, for $t>1$, we can rewrite $p_t^\star$ as the conditional density of $X_t$ given $X_1^{t-1}$ and $p_1^\star$ as the density of $X_1$ (in this case $p_1^\star$ is deterministic). To emphasize this fact, in this case we note $p_t^\star(.|X_1^{t-1})$ instead of just $p_t^\star$. Note that the density of the vector $X_1^n$ with respect to $\mu^{\otimes n}$ is therefore
\begin{equation}
\label{eq:dens_prod}
x_1^n\mapsto \prod_{i=1}^n p^\star_t(x_t|X_1^{t-1}=x_1^{t-1}).
\end{equation}

In the even simpler case that the coordinates of $X_1^n$ are independent, for all $t\geq 1$, $p^\star_t$ is the (deterministic)
density of $X_t$ w.r.t. to $\mu$.

We might be interested by an enlargement of the filtration. For instance, the conditional distribution of $X_t$ can be a function of not only the past realizations of $X_t$, but also of additional covariates: imagine that each step $t$, the distribution of $X_t$ depends on the past but also on an observed variable $C_t$, so that the conditional densities can be written
\begin{equation}\label{eq:dens_prod_contexte}
 p^\star_t(x_t) = p^\star_t(x_t | X_1^{t-1}, C_1^t).
\end{equation}
Every decent model of evolution for $X_t$ depends on $X_1^{t-1}$ but also on $C_1^t$. The natural filtration in this setting is $\Fcal_{t-1}=\sigma(X_1^{t-1}, C_1^t)$.

Another setup, where the filtration might not be completely observed, is chains with infinite memory that can model potentially infinite neuronal networks. Even if we do not observe the chain on the whole network with infinite past, we might be interested by approximating $p^\star_t$, which exists as a function of the infinite network, with functions that only involve a smaller subset of observed neurons (see Section \ref{sec:infinite_chain} for more details).

\subsection{Models and penalized (partial) log-likelihood}

Consider a sequence of models $(\{p^m_\theta : \theta \in \Theta_m \subset \Rbb^{D_m}\})_{m \in \Mcal}$ for some countable set $\Mcal$. Each $p^m_\theta$ is a sequence $p^m_\theta=(p^m_{\theta,t})_{t\in [n]}$, with $p^m_{\theta,t}$ being a candidate at being $p^\star_t$. In particular, the candidate $p^m_{\theta,t}$ must be predictable.

For any $m \in \Mcal$, define the (partial) log-likelihood of parameter $\theta \in \Theta_m$ given the observations by
\begin{equation*}
\ell_n(\theta)
 = \sum_{t=1}^n \log p^m_{\theta,t}(X_t).
\end{equation*}
Note that in the case mentioned above where $\Fcal_t=\sigma(X_1^t)$ for $t\geq 1$ and $\Fcal_0$ is the trivial $\sigma$-algebra, because of~\eqref{eq:dens_prod}, $\ell_n(\theta)$ is exactly the log-likelihood $\log p^m_\theta(X_1^n)$. For models as in~\eqref{eq:dens_prod_contexte}, this partial log-likelihood still has the convenient form of a conditional log-likelihood: $\ell_n(\theta) = \log p^m_\theta(X_1^n | C_1^n)$, and it matches Cox's partial likelihood \cite{cox1975partial}. This might no longer be the case for larger filtrations.

For each $m \in \Mcal$, define the maximum likelihood estimator of model $m$ by
\begin{equation*}
 \hat{\theta}^m \in \underset{\theta \in \Theta_m}{\argmax} \; \frac{1}{n} \ell_n(\theta).
\end{equation*}

Finally, take a penalty $\pen : \Mcal \to \Rbb_{+}$ and select a model $\hat{m}$ that minimizes the penalized log-likelihood:
\begin{equation*}
\widehat{m} \in \underset{m \in \Mcal}{\argmin} \left(-\frac{1}{n} \ell_n(\hat{\theta}^m) + \pen(m) \right).
\end{equation*}

The penalized likelihood estimator of $p^\star$ is therefore $\tilde{p}=p^{\hat{m}}_{\hat{\theta}^{\hat{m}}}$. We aim to understand the properties of approximation of $\tilde{p}$ with respect to $p^\star$.

\subsection{Stochastic loss function}

Classical approaches \cite{massart2007concentration, spok2012,spok2017}, generally use an expectation of the contrast to define the loss function. For instance, in i.i.d. examples, the log-likelihood is naturally linked to the Kullback-Leibler divergence between the distributions.

Here, we want to keep the inherent martingale structure that comes with the filtration and with our object of interest $p^\star$. This is why we are using the stochastic loss function $\Kbf_n$ defined as follows: for any sequence of conditional densities $p = (p_t)_{t \in [n]}$, let
\begin{equation*}
\Kbf_n(p)
 = \frac{1}{n} \sum_{t=1}^n \Ebb\left[ \log \frac{p_t^\star(X_t)}{p_t(X_t)} \, \Big| \Fcal_{t-1} \right].
\end{equation*}

This can be seen as the 
mean of the conditional Kullback-Leibler divergence in the sense that
\begin{equation*}
 \Ebb\left[ \log \frac{p_t^\star(X_t)}{p_t(X_t)} \, \Big| \Fcal_{t-1} \right]
\end{equation*}
is a predictable quantity which corresponds to the Kullback-Leibler divergence between the distributions with densities $p^\star_t$ and $p_t$ w.r.t $\mu$, conditionally to $\Fcal_{t-1}$.

Note that in the case where $\Fcal_t=\sigma(X_1^t)$ for $t\geq 1$ and $\Fcal_0$ is the trivial $\sigma$-algebra, because of \eqref{eq:dens_prod}, $n\Ebb[\Kbf_n(p)]$ is exactly the Kullback-Leibler divergence between the distributions defined respectively by $p^\star$ and $p$.

Our oracle inequalities presented below are bounds on the stochastic loss $\Kbf_n(\tilde{p})$ in probability and in expectation respectively.


\section{Main results \label{sec:main}}

\subsection{Assumptions}
Let us first discuss the main assumptions.

Because we are using a Kullback-Leibler-like divergence as a loss, we need to ensure that it does not diverge. A natural assumption in this respect is to assume that $p^\star$ and the candidates $p^m_\theta$ stay far from $0$ and do not explode. This is done almost surely in Assumption~\ref{hyp_tails_bounded} and with high probability only in Assumption~\ref{hyp_tails}.

\begin{customhyp}{1}\label{hyp_tails_bounded} There exists $\varepsilon > 0$ such that a.s., for all $t \in [n]$, 
$p^\star_t(X_t) \in [\varepsilon, \varepsilon^{-1}]$ and for all $m \in \Mcal$ and all $\theta \in \Theta_m$, 
$p^m_{\theta,t}(X_t) \in [\varepsilon, \varepsilon^{-1}]$. Without loss of generality, we assume $\log \varepsilon < -1$.
\end{customhyp}

A weaker version of this assumption is the following.

\begin{customhyp}{1bis}\label{hyp_tails}
For all $m \in \Mcal$, there exists a finite constant $B_m$ such that for all $y \geq 1$,
\begin{equation*}
 \Pbb\left( \exists t \in [n], \ \exists m \in \Mcal, \ \sup_{\delta, \theta \in \Theta_m \cup \{\star\}} \left| 
 \log \frac{p^m_{\delta,t}(X_t)}{p^m_{\theta,t}(X_t)}
 \right| 
 > B_m y \, \Big| \, \Fcal_{t-1} \right) \leq e^{-y} \quad \text{a.s.},
\end{equation*}
with the convention $p^m_\star = p^\star$. Without loss of generality, we assume $B_m \geq 1$ for all $m \in \Mcal$.
\end{customhyp}

\bigskip

The second category of assumptions replaces the exponential family assumption of the classic asymptotic results on MLE. It states that the parameterization of the models is Lipschitz w.r.t. some norm that is bounded on $\Theta_m$.

\begin{customhyp}{2}\label{hyp_lipschitz_bounded} For all $m \in \Mcal$, there exist a norm $\| \cdot \|_m$ on $\Rbb^{D_m}$ and finite, positive constants $L_m$ and $M_m$ such that a.s., for all $t \in [n]$, for all $m \in \Mcal$ and all $\delta, \theta \in \Theta_m$,
\begin{equation*}
\left| \log \frac{p^m_{\delta,t}(X_t)}{p^m_{\theta,t}(X_t)} \right|
 \leq L_m \|\delta - \theta\|_{m}
\end{equation*}
and
\begin{equation*}
 \|\delta - \theta\|_{m} \leq M_m.
\end{equation*}
Without loss of generality, we assume $L_m M_m \geq 1$ for all $m \in \Mcal$.
\end{customhyp}

Likewise, a weaker version of this assumption is

\begin{customhyp}{2bis}\label{hyp_lipschitz} For all $m \in \Mcal$, there exist a norm $\| \cdot \|_m$ on $\Rbb^{D_m}$ and finite, positive constants $L_m$ and $M_m$ such that
\begin{equation*}
\Pbb\left(
 \exists t \in [n], \ \exists m \in \Mcal, \ \sup_{\delta, \theta \in \Theta_m}
 \left| \log \frac{p^m_{\delta,t}(X_t)}{p^m_{\theta,t}(X_t)} \right|
 > L_m \|\delta - \theta\|_m \log n
\right) \leq n^{-1}
\end{equation*}
and
\begin{equation*}
 \|\delta - \theta\|_{m} \leq M_m.
\end{equation*}
Without loss of generality, we assume $L_m M_m \geq 1$ for all $m \in \Mcal$.
\end{customhyp}

Finally we will need in the proof to define a supremum of $x \mapsto p^\star_t(x)$ and $p^m_{\theta,t}(x)$ over $\Xcal$. The following assumption ensures that we can define it in a way that is measurable w.r.t. the filtration $\Fcal_{t-1}$.

\begin{customhyp}{3}\label{densites_continues} There exists a countable dense subset $\Qcal$ of $\Xcal$ such that almost surely, for all $t \in [n]$, $m \in \Mcal$ and $\theta \in \Theta_m$,
\begin{equation*}
\sup_{x \in \Qcal} p^m_{\theta,t}(x) = \sup_{x \in \Xcal} p^m_{\theta,t}(x)
\end{equation*}
and
\begin{equation*}
\sup_{x \in \Qcal} p^\star_t(x) = \sup_{x \in \Xcal} p^\star_t(x).
\end{equation*}
\end{customhyp}

Note that if $\Xcal$ is discrete and countable, this assumption is automatically verified. Likewise, since $\Xcal$ is separable (because it is assumed to be Polish), if almost surely, for all $m \in \Mcal$ and $\theta \in \Theta_m$, the (random) functions $x \mapsto p^\star_t(x)$ and $x \mapsto p^m_{\theta,t}(x)$ are continuous, this assumption holds.

With this small set of assumptions, we can now state our oracle inequalities.

\subsection{Oracle inequality, bounded case}

\begin{theorem}\label{generalmodelselectiontheorem_bounded}
Assume that $n \geq 2$ and that Assumptions~\ref{hyp_tails_bounded}, \ref{hyp_lipschitz_bounded} and~\ref{densites_continues} hold. There exist positive numerical constants $C$ and $C'$ such that the following holds.
For each $m \in \Mcal$, let $A_m = L_m M_m + 2 \log(\varepsilon^{-1})$, and assume that
 \begin{equation*}
\Sigma = \sum_{m \in \Mcal} \log(A_m) e^{-D_m} < +\infty .
\end{equation*}
Let $\kappa \in (0,1]$.
If for all $m \in \Mcal$,
\begin{equation*}
 \pen(m) \geq \frac{C}{\kappa} A_m^2 \log(\varepsilon^{-1})^{3/2} \log(nA_m)^2 \frac{D_m}{n},
\end{equation*}
then for all $x \geq 0$, with probability at least $1 - 18 \log(n) \Sigma e^{-x}$,
\begin{multline*}
(1 - \kappa) 
\Kbf_n(\tilde{p})
    \leq \inf_{m \in \Mcal} \left( (1 + \kappa) \inf_{\theta \in \Theta^{D_m}} \Kbf_n(p_{\theta}^{m})
    + 2 \pen(m)
    \right) \\
    + \frac{C'}{\kappa} \left(
        A_m \log(\varepsilon^{-1})^{3/2} \log(nA_m)^2
        + A_{\hat{m}} \log(\varepsilon^{-1})^{3/2} \log(nA_{\hat{m}})^2
    \right)\frac{x}{n}.
\end{multline*}
\end{theorem}

\begin{proof}
 See Section~\ref{sectionproofgeneralmodelselection}.
\end{proof}

This result is an oracle inequality in probability. The penalty term is proportional to $D_m/n$, as in classical oracle inequalities for nested or not too complex families of models. This is ensured by the summability condition on $\Sigma$ (see \cite{birge-massart, massart2007concentration} for instance for a discussion about the complexity of a family of models). Hence, it can be read as a usual oracle inequality: the loss is--up to a constant factor and a residual term--smaller than the best bias-variance trade-off in the family of models, with a variance which is of order $D_m/n$.

Due to the generality of the result, the penalty is actually a bit larger than in the original AIC criterion, with additional logarithmic factors, and depends on the lower bound $\varepsilon$ and the Lipschitz constants $L_m$ and $M_m$. Note the presence of a term depending on $\hat{m}$ in the residual term: due to the fact that $\Mcal$ might be infinite, additional assumptions are required to get rid of it, such as the uniform bound on $(A_m)_{m \in \Mcal}$ introduced in the following corollary.

The next corollary gives a result in expectation under a slightly more restrictive assumption (when $\sup_m L_m M_m$ is bounded), and is proved in Section~\ref{sec_proof_corollaries}.

\begin{corollary}
\label{cor_generalmodelselection_bounded}
Under the assumptions and with the same constants and notations of Theorem~\ref{generalmodelselectiontheorem_bounded}, if there exists $A(n)$ such that
\begin{equation*}
\sup_{m \in \Mcal} A_m \leq A(n) , 
\end{equation*}
then
\begin{multline*}
(1 - \kappa) \Ebb \left[\Kbf_n(\tilde{p})\right]
 \leq \Ebb\left[ \inf_{m \in \Mcal} \left( (1 + \kappa) \inf_{\theta \in \Theta^{D_m}} \Kbf_n(p_{\theta}^{m})
 + 2 \pen(m)
 \right) \right] \\
 + \frac{36 C'}{\kappa} \Sigma A(n) \log(\varepsilon^{-1})^{3/2} \log(nA(n))^2 \frac{\log n}{n}.
\end{multline*}
\end{corollary}

\subsection{Oracle inequality, unbounded case}

The bounded case might be a bit restrictive for some applications, so we can relax the assumptions, up to additional logarithmic factors.

\begin{theorem}\label{generalmodelselectiontheorem}
Assume that $n \geq 2$ and that Assumptions~\ref{hyp_tails}, \ref{hyp_lipschitz} and~\ref{densites_continues} hold. There exist positive numerical constants $C$ and $C'$ such that the following holds.
For each $m \in \Mcal$, let $A_m = L_m M_m + B_m$, and assume that
 \begin{equation*}
\Sigma = \sum_{m \in \Mcal} \log(A_m) e^{-D_m} < +\infty
\end{equation*}
Let $\kappa \in (0,1]$.
If for all $m \in \Mcal$,
\begin{equation*}
 \pen(m) \geq \frac{C}{\kappa} A_m^2 B_m^{3/2} \log(n)^{7/2} \log(nA_m)^2 \frac{D_m}{n},
\end{equation*}
then for all $x \geq 0$, with probability at least $1 - 2n^{-1} - 18 \log(n) \Sigma e^{-x}$,
\begin{multline*}
(1 - \kappa) \Kbf_n(\tilde{p})
    \leq \inf_{m \in \Mcal} \left( (1 + \kappa) \inf_{\theta \in \Theta^{D_m}} \Kbf_n(p_{\theta}^{m})
    + 2 \pen(m)
    \right) \\
    + \frac{C'}{\kappa} \left(
        A_m B_m^{3/2} \log(nA_m)^2
        + A_{\hat{m}} B_{\hat{m}}^{3/2} \log(nA_{\hat{m}})^2
        \right) \frac{(\log n)^{5/2} x}{n}.
\end{multline*}
\end{theorem}

\begin{proof}
 See Section~\ref{sectionproofgeneralmodelselection}.
\end{proof}

Note that we may replace the term $2n^{-1}$ in the probability by $2n^{-\alpha}$ for any $\alpha \in [1,n]$, provided Assumption~\ref{hyp_lipschitz} holds with probability $n^{-\alpha}$ for a bound $\alpha L_m \|\delta - \theta\|_m \log n$. This changes the constants $C$ and $C'$ of Theorem~\ref{generalmodelselectiontheorem} into $C \alpha^{7/2}$ and $C' \alpha^{5/2}$ respectively.

Except for these extra logarithmic factors, the previous result is essentially the same as Theorem \ref{generalmodelselectiontheorem_bounded}. An expectation version of this result holds under similar assumptions on the family of models. Its proof can be found in Section~\ref{sec_proof_corollaries}.

\begin{corollary}
\label{cor_generalmodelselection}
Under the assumptions and with the same constants and notations of Theorem~\ref{generalmodelselectiontheorem}, if there exist $A(n)$ and $B(n)$ such that
\begin{equation*}
\sup_{m \in \Mcal} A_m \leq A(n)
 \quad \text{and} \quad
 \sup_{m \in \Mcal} B_m \leq B(n) ,
\end{equation*}
then
\begin{multline*}
(1 - \kappa) \Ebb \left[\Kbf_n(\tilde{p})\right]
 \leq \Ebb\left[ \inf_{m \in \Mcal} \left( (1 + \kappa) \inf_{\theta \in \Theta^{D_m}} \Kbf_n(p_{\theta}^{m})
 + 2 \pen(m)
 \right) \right] \\
 + \frac{40 C'}{\kappa} \Sigma A(n) B(n)^{3/2} \log(nA(n))^2 \frac{(\log n)^{7/2}}{n}.
\end{multline*}
\end{corollary}


\subsection{Where does the penalty in \texorpdfstring{$D_m/n$}{Dim/n} comes from ?}

The common crucial point to all non-asymptotic controls of log-likelihood estimators \cite{massart2007concentration, castellan, spok2012, spok2017} lies in the control of the recentered contrast at the estimation point. In our framework, with the above notation, it means controlling
\begin{equation*}
\nu(H) = \frac{1}{n}\sum_{t=1}^n [H_t(X_t)-\Ebb(H_t(X_t)|\Fcal_{t-1})],
\end{equation*}
where $H_t$ is equal to 
\begin{equation*}
H_t
 = H_{\theta,t}^m
 = -\log\left(\frac{p^m_{\theta,t}(X_t)}{p^\star_t(X_t)}\right).
\end{equation*}

As long as $\theta$ is fixed and deterministic, this is a martingale and various exponential tail bounds are applicable. 
Without going into details right now, let us just say that with high probability,
\begin{equation}\label{eq:one_nu}
\nu(H)=\Ocal\left(\sqrt{V(H)/n}\right)
\end{equation}
with
\begin{equation*}
V(H)=\frac{1}{n} \sum_{t=1}^n\Ebb( H_t(X_t)^2|\Fcal_{t-1}).
\end{equation*}
In particular in the i.i.d. framework, where $V(H)$ is deterministic, of the order of the variance of $H_1$,
we recover what the central limit theorem entails, that is, the fluctuation of the empirical process are of order $n^{-1/2}$ multiplied by the standard deviation. In more general settings, $V(H)$ is random: it is to the bracket of the martingale. It is usually necessary to restrict oneself to an event where $V(H)$ is bounded to obtain such non asymptotic control of $\nu(H)$ (see for instance \cite{bercu2008}).

This is still not sufficient to conclude: to understand what happens for the maximum likelihood estimator, we need to control $\nu(H^m_{\hat{\theta}_m,t})$. The fact that $\hat{\theta}_m$ depends on the whole trajectory prevents us from using classical inequalities for centered processes.

\paragraph{Talagrand inequality and Wilks phenomenon in the i.i.d. setting}
In the classical i.i.d. setting, and as a first approach, Talagrand's inequality (see \cite{massart2007concentration} for various model selection contexts) leads to this kind of control:
\begin{equation*}
\sup_{\theta \in \Theta_m} \frac{\nu(H_{\theta}^m)}{\|\theta\|}
 = \Ocal\left(\Ebb\left(\sup_{\theta \in \Theta_m, \|\theta\|\leq 1} \nu(H_{\theta}^m )\right)\right)
 + \Ocal\left(\sqrt{\frac{\sup_{\theta \in \Theta_m, \|\theta\|\leq 1}V(H_\theta^m)}{n}}\right).
\end{equation*}

It turns out that in many models used in \cite{massart2007concentration}, $\nu(H_{\theta}^m)$ is in fact linear or close to linear with respect to $\theta$, so that the first term is of order $\sqrt{D_m/n}$, the dimension of model $m$. In particular, this tells us non asymptotically that
\begin{equation*}
\nu(H_{\hat{\theta}_m}^m) = \Ocal\left( \|\hat{\theta}_m\|\sqrt{D_m/n}\right).
\end{equation*}
Beyond the i.i.d. setting, similar results can be obtained by replacing Talagrand's inequality with Baraud's inequality \cite{baraud}.

In comparison, the Wilks phenomenon predicts that the order of magnitude of $\nu(H_{\hat{\theta}_m}^m)$ is in $D_m/n$, and not $\sqrt{D_m/n}$ (at least if the model $m$ is well specified, that is there exists $\theta^*_m$ such that $p^\star=p_{\theta^*_m}^m$): this non asymptotic Talagrand-like bound is pessimistic. But the trick above with the linearity with respect to $\theta$ shows us that this pessimistic bound can hopefully be multiplied by the norm of $\hat{\theta}_m$. Per se this is not much, but with a careful choice of the set over which the supremum is taken, we can replace $\|\hat{\theta}_m\|$ by the distance to a pivot $\theta^*_m$ such that, approximately,
\begin{equation*}
\nu(H_{\hat{\theta}_m}^m)
 =\Ocal\left( \|\hat{\theta}_m-\theta^*_m\|\sqrt{D_m/n}\right)
 \leq \Ocal \left(\|\hat{\theta}_m-\theta^*_m\|^2 \right)
 + \Ocal \left(\frac{D_m}{n}\right).
\end{equation*}
It remains to be in a nice enough setting (such as Lipschitz parameterization of $\theta \mapsto p_{\theta}^m$) that $\Ocal \left(\|\hat{\theta}_m-\theta^*_m\|^2 \right)$ is a small fraction of the loss function we are using (here $\Kbf_n(p^m_{\hat{\theta}_m})$). All in all, up to constant, the remainder that the penalty is meant to account for is indeed in $D_m/n$ up to multiplicative constants. 

Before going further, let us make a small remark about the pivot $\theta^*_m$. If the model $m$ is well specified, it easy to use the $\theta^*_m$ such that $p^\star=p_{\theta^*_m}^m$. If this is not the case, usually people use the best approximation of $p^\star$ by a $p_\theta^m$ for a certain loss function (say the Kullback-Leibler divergence). So in general, and even in non i.i.d. cases, the classical way to choose a pivot is deterministic, because the loss that is used is deterministic as well.

\paragraph{Non asymptotic Wilks-like results}

In various models, the approach given above cannot work as nicely as it mainly relies on linearity, so the brute force concentration inequality "à la Talagrand" is in fact too pessimistic. What we would like, in an ideal and over-simplified world, is that \eqref{eq:one_nu} holds with a variance term that would be directly $V(H_{\hat{\theta}_m}^m)$, without having to take (and pay for) the supremum over all $H^m_\theta$.

In \cite{massart-wilk}, the authors restrict the previous supremum to nice small balls, so that one gets the correct behavior in $D_m/n$ for $\nu(H_{\hat{\theta}_m}^m)$ directly. This can be especially useful when one wants to produce very sharp constants in the penalty \cite{arlot}.

Another way to think about this, is to provide an upper function of the process, that is  to prove that 
\begin{equation*}
\sup_{\theta \in \Theta_m} \left[ \nu(H_{\theta}^m)- \Ocal\left(\sqrt{V(H_{\theta}^m)/n}\right) \right]
\end{equation*}
is negative with high probability. This is Spokoiny's approach \cite{spok2012, spok2017}, whose statistical results are the closest approach to ours for their generality, even if the author does not perform model selection per se.

\paragraph{Self-normalized martingales}

Note that it is far from obvious to obtain an inequality such as~\eqref{eq:one_nu} when $V(H)$ is random. It is possible for martingales, in which case $V(H)$ is  usually a random quantity called the bracket of the martingale.
Exponential inequalities for martingales have been developed, especially for point processes where one can go from a control of the martingale with deterministic upper bound on the bracket, which already tells a lot on the properties of the maximum likelihood estimator \cite{vandeGeer}, to a control where the bracket of the martingale can replace the deterministic upper bound in the deviation \cite{Hansen_2015} (up to some corrections). 

In the same line, many works on self-normalization of martingales try to directly control the ratio of the martingale by its bracket \cite{bercu2008,pena-book}. To our knowledge, nothing exists in this direction for a supremum of the ratios (with random renormalization) except our recent result \cite{aubert:hal-04526484}.

\paragraph{Deterministically renormalized empirical process}

From a more deterministic point of view, several works aim to choose the correct deterministic renormalization of the empirical process in an i.i.d. setting. There are two main ways to use it for model selection. 

On the one hand, if the empirical process itself, once renormalized by a deterministic quantity of the form $d^2(\theta,\theta^*)$ \footnote{for a nice deterministic $d$ distance on $\Theta_m$, that can be linked to $V(H_\theta^m)$, up to possible additional terms}, satisfies a nice exponential inequality, then this can be chained to either directly obtain a Talagrand-like concentration on the supremum \cite{baraud}, or used to get nice upper-functions \cite{spok2012}.

On the other hand, one can refine the renormalization inside the supremum and replace 
\begin{equation*}
\sup_{\theta \in \Theta_m} \frac{\nu(H_{\theta}^m)}{\|\theta\|}
 \quad \mbox{by} \quad
 \sup_{\theta \in \Theta_m} \frac{\nu(H_{\theta}^m)}{d^2(\theta,\theta^*)+x^2},
\end{equation*}
for a positive $x$ to be chosen later. It is with this approach that~\cite{massart2007concentration} proposed a fairly general approach to model selection with penalty proportional to $D_m$ even in non linear settings. In this setting, we still do not obtain the order of magnitude of the Wilks phenomenon directly, but because, as previously, $d^2(\theta,\theta^*)$ is close to the loss, we can still get a penalty in $D_m/n$. The huge advantage is that we can go further away from the linearity assumption by adding more flexibility inside the family of models thanks to the use of $d^2(.,.)$ instead of the Euclidean norm.

\paragraph{Our method and the difference with existing works}

In our case, even the loss is stochastic in the most general case. Only the martingale structure remains. As a consequence, we do not have access to a deterministic $d^2(.,.)$ and we cannot even properly define a nice deterministic pivot $\theta^*_m$ if the model is misspecified. However, we still have access to the bracket of the martingale $V(H_{\theta})$ and it turns out that $V(H_{\hat{\theta}_m}^m)$ is nicely comparable to $\Kbf_n(p_{\hat{\theta}_m}^m)$. We recently proved a concentration inequality for the supremum of stochastically normalized processes \cite{aubert:hal-04526484}, so that we can use more or less Massart's above argument after replacing $d^2(.,.)$ by the bracket of the martingales. The problem of the pivot is solved by working directly in the space of probability and using $p^\star$ as pivot, thus bypassing the misspecification issue.

Finally, let us compare our set of assumptions to \cite{spok2012, spok2017} which constitutes, to our knowledge, the most general non asymptotic results on maximum likelihood estimation. Spokoiny is using the Kullback-Leibler divergence as a reference. It allows him to use a deterministic pivot defined as the closest point to the truth inside the model for the deterministic distance \cite{spok2012} or with additional quadratic corrections \cite{spok2017}. His main assumptions rely on renormalized exponential inequalities on the gradient of the likelihood in a neighborhood of the pivot, where the normalization is quadratic. It allows him to prove a quadratic-like behavior for $\nu(H_{\hat{\theta}_m}^m)$, in this sense proving the non asymptotic Wilks phenomenon rather sharply, which we cannot do with our method.

In contrast, our method is applicable in more general settings: we do not need to use a deterministic loss, nor deterministic parametric pivots, or to assume that the log-likehood is differentiable or that its gradient satistifies exponential inequalities. We only assume that the parameterization is Lipschitz. By martingales properties, the log-likelihood (and not its gradient) automatically satisfies \eqref{eq:one_nu}, which is the key to prove the results on the supremum. In the end, we obtain a generalized AIC criterion (that is, a penalty proportional to $D_m/n$) and prove non asymptotic oracle inequalities in settings where none of the existing work applies.

\section{Applications \label{sec:appli}}

The previous oracle inequalities are very general. This section aims to explicit how they compare to existing results. Let us look at their applications in various settings.

\subsection{The i.i.d. sample case}

Let us begin with the original AIC setting in its i.i.d. format.
Let $X_1, \dots,X_n$ be i.i.d. real valued random variables with density $p^\star_1$. We consider the simple case where the filtration is generated by the observations, so that all the $p^\star_t$ are deterministic and equal to $p^\star_1$.

We also assume that under each model, the variables $X_1, \dots, X_n$ are i.i.d., so that $p^m_{\theta,t} = p^m_{\theta,1}$ for all $t$ and these functions are deterministic.

In this case, $\Kbf_n(p^m_\theta)$ is directly $\DKL(p^\star_1 d\mu, p^m_{\theta,1}d\mu)$, the Kullback-Leibler divergence between the distribution of $X_1$ and the distribution with density $p^m_{\theta,1}$.

\paragraph{Validation of the assumptions}

Assumption~\ref{hyp_tails_bounded} is classical in this setting, see e.g. \cite{massart2007concentration}, at least for the lower bound. In \cite{massart2007concentration} or \cite{castellan}, there is no upper bound assumption but there is a twist: it is not the Kullback-Leibler divergence that is controlled but the Hellinger distance. Here we can also relax the lower bound assumption by using Assumption~\ref{hyp_tails}.

Assumption~\ref{hyp_lipschitz_bounded}--the $L_m$ part--can be a consequence of the fact that $p^m_{\theta,1}$ is Lipschitz with constant $L_m \varepsilon^{-1}$, or of directly assuming that $\log(p^m_{\theta,1})$ is Lipschitz with constant $L_m$. It can be relaxed with Assumption~\ref{hyp_lipschitz}.
Classical asymptotic results (Wilks phenomenon or even just the consistency of the MLE) require very strong differentiability assumptions, that are not needed here.
Note that the bound $M_m$ in Assumptions~\ref{hyp_lipschitz_bounded} and~\ref{hyp_lipschitz} entails that the models $\Theta_m$ are compact, which is assumed in general to obtain consistency in M-estimation, whether explicitely or implicitely by assuming that the estimator converges to some limit, see e.g.~\cite{vdw1996Mestimators}.


The classical model selection "à la Birgé-Massart" for densities assumes that the models are close to linear to work, so it seems that this is incompatible with our boundedness assumption \cite{massart2007concentration}. However this is not the case because we are forcing the parametrization to be a density that satisfies Assumption~\ref{hyp_tails_bounded}. To illustrate this, if we want to compare our assumptions with theirs, let us restrict ourselves to a well known case: the histogram selection \cite{massart2007concentration,castellan}. There, $X_1,...,X_n$ are i.i.d. with density $p^\star_1$ with respect to the Lebesgue measure on $[0,1]$ and the model $m$ is based on a partition with $D_m$ intervals of $[0,1]$ of equal length. Then, for $\theta\in \Theta_m \subset \Rbb^{D_m}$,
\begin{equation*}
p_{\theta,1}^m = \sum_{I\in m} \theta_I {\bf 1}_I.
\end{equation*}
Since this must be a density, the model is in fact
\begin{equation*}
\Theta_m=\left\{\theta=(\theta_I)_{I\in m} \mbox{ such that for all } I, \varepsilon\leq \theta_I \leq \varepsilon^{-1} \mbox{ and } \sum_{I\in m} \frac{\theta_I}{D_m} =1\right\}.
\end{equation*}
In this particular case, $p^m_{\theta,1}$ is Lipschitz with constant $L_m=1$ w.r.t. $\|\theta\|_\infty$ and $M_m=\varepsilon^{-1}-\varepsilon$ is an upper bound of the diameter of $\Theta_m$ for this norm.

Assumption~\ref{densites_continues} is automatically fulfilled with $\mathcal{Q}=\mathbb{Q}\cap[0,1]$, the set of rational numbers.

\paragraph{Result}
Following Theorem \ref{generalmodelselectiontheorem_bounded} and its Corollary, $A_m$ does not depend on $m$ anymore. Therefore, our oracle inequality holds as soon as
\begin{equation*}
pen(m)=\Ocal\left( (\log n)^2 \frac{D_m}{n}\right).
\end{equation*}
Our penalty is larger than the one in \cite{massart2007concentration} due to extra logarithmic factors and looks more like a BIC criterion. At this price we are able to prove an oracle inequality directly on the Kullback-Leibler divergence, instead of a mixed oracle inequality involving both the Kullback-Leibler divergence and the Hellinger distance.

\subsection{Hidden Markov models}

A hidden Markov model is a stochastic process $(H_t, X_t)_t$ where only the observations $(X_t)_t$ are observed, such that the hidden process $(H_t)_t$ is a Markov chain and such that the $X_s$ are independent conditionally to $(H_t)_t$ with a distribution depending only on the corresponding $H_s$. These models have been widely used since their formalization by~\cite{baum1966hmm}, as they are able to account for complex dependencies in time processes while keeping a very simple and easily interpretable structure.

In this section, we consider finite state space hidden Markov model, in which the process $(H_t)_t$ takes values in a finite space $\Hcal$. The parameters of a hidden Markov model are the initial distribution $\pi$ and the transition kernel $Q$ of the hidden process $(H_t)_t$ as well as the emission densities, that is the family $\nu = (\nu_h)_{h \in \Hcal}$, where $\nu_h$ is the density of the distribution of $X_1$ conditionally to $H_1 = h$ w.r.t. the Lebesgue measure. The initial distribution $\pi$ typically cannot be exactly recovered, so, in general, the parameters we wish to estimate are $Q$ and $\nu$.

The closest result to ours in this setting is the one from~\cite{lehericy2021hmm_mle}, who proves an oracle inequality for a maximum likelihood estimator on misspecified hidden Markov models. Their assumptions and proofs are similar to ours, although they rely on tools that are specific to hidden Markov models to obtain their oracle inequality.

\paragraph{Validation of the assumptions}

In what follows, we only assume that the models are HMM, while the true distribution of $(X_t)_t$ may not be one. As such, we treat it separately, before introducing the models.

Concerning Assumption~\ref{hyp_tails} for the true distribution, the lower tails of $\log p^\star_t(X_t)$ are always automatically sub-exponential by direct application of Markov's inequality. The control of the upper tails follows from the assumption that the conditional densities of $(X_t)_t$ admit a finite moment, that is, there exists constants $\delta > 0$ and $M_\delta > 0$ such that almost surely,
\begin{equation}
\label{eq_HMM_hyp_tails_star}
\sup_{t \in [n]} \Ebb[ p^\star_t(X_t)^\delta \, | \, \Fcal_{t-1} ] \leq M_\delta < +\infty.
\end{equation}
Under this assumption, there exists $B^* > 0$ depending on $\delta$ and $M_\delta$ such that for all $y \geq 1$ and $t \in [n]$,
\begin{equation*}
\Pbb( |\log p^\star_t(X_t)| > B^* y \, | \, \Fcal_{t-1} ) \leq e^{-y}.
\end{equation*}

Let us now introduce the models. Let $C_Q > 0$ and $\alpha \geq 1$. For all $m$, let $h_m \in \Nbb^*$ be the number of hidden states of model $m$. Let $S_m = \{ g_\eta, \eta \in E_m \subset \Rbb^{e_m} \}$ be a parametric set of probability densities on $\Xcal$ such that $g_\eta \geq n^{-\alpha}$ for all $\eta$. The model $\Theta_m$ is defined as the set
\begin{align*}
\Theta_m = \Big\{
    (a, q, \eta) \in \Rbb^{h_m-1} \times & \Rbb^{h_m(h_m-1)} \times (E_m)^{h_m} \ \text{ s.t. } \\ 
        (C_Q \log n)^{-1} &\leq h_m a(i) \leq C_Q \log n \quad \text{ for all } i \in [h_m-1], \\
        (C_Q \log n)^{-1} &\leq h_m q(j,i) \leq C_Q \log n \quad\text{ for all } i \in [h_m-1] \text{ and } j \in [h_m], \\ 
    (C_Q \log n)^{-1} &\leq h_m \left(1 - \sum_{j=1}^{h_m-1} a(j)\right) \leq C_Q \log n \\
    \text{ and } (C_Q \log n)^{-1} &\leq h_m \left(1 - \sum_{j=1}^{h_m-1} q(i,j) \right) \leq C_Q \log n \quad \text{ for all } i \in [h_m]
    \Big\}.
\end{align*}
This model is of dimension $D_m = h_m e_m + h_m^2 - 1$, and the parameter $\theta = (a, q, \eta) \in \Theta_m$ is associated to the HMM parameters
\begin{equation*}
\left\{
\begin{aligned}
&\pi^m_\theta(i) = a(i) \qquad \text{for all } i \in [h_m-1], \\
&\pi^m_\theta(h_m) = 1 - \sum_{i=1}^{h_m-1} a(i), \\
&Q^m_\theta(i,j) = q(i,j) \qquad \text{for all } j \in [h_m-1] \text{ and } i \in [h_m], \\
&Q^m_\theta(i,h_m) = 1 - \sum_{j=1}^{h_m-1} q(i,j) \qquad \text{for all } i \in [h_m], \\
&\nu^m_{\theta,i} = g_{\eta_i} \qquad \text{for all } i \in [h_m].
\end{aligned}
\right.
\end{equation*}
The inequalities on the parameters $a$ and $q$ ensure that all entries of the initial distribution and transition matrices are between $(C_Q \log n)^{-1} h_m^{-1}$ and $(C_Q \log n) h_m^{-1}$.
Finally, the likelihood of the observations $X_1^n$ under the parameter $\theta \in \Theta_m$ is
\begin{equation*}
p^m_\theta(X_1^n) = \sum_{(i_1, \dots, i_n) \in [h_m]^n} \pi^m_\theta(i_1) \nu^m_{\theta, i_1}(X_1) \prod_{t=2}^n Q^m_\theta(i_{t-1}, i_t) \nu^m_{\theta, i_t}(X_t).
\end{equation*}

Given the upper and lower bounds on the initial distribution and transition matrix, Assumptions~\ref{hyp_tails_bounded} and \ref{hyp_tails} can be replaced by an assumption on the average $\bar{\nu}^m_{\theta} = \frac{1}{h_m} \sum_{i \in [h_m]} \nu^m_{\theta,i}$ of the emission densities $(\nu^m_{\theta,i})_{i \in [h_m]}$, since
\begin{align*}
p^m_{\theta,t}(X_t)
 &= \sum_{i,i' \in \Hcal} p^m_{\theta}(H_{t-1} = i | X_1^{t-1}) Q^m_\theta(i,i') \nu^m_{\theta,i'}(X_t) \\
 &\in \left[ (C_Q \log n)^{-1} \bar\nu(X_t) , (C_Q \log n) \bar{\nu}^m_{\theta}(X_t) \right].
\end{align*}
Thus, to check Assumption~\ref{hyp_tails}, we will control the tails of $\bar\nu^m_\theta(X_t)$. Since all the emission densities are lower bounded by $n^{-\alpha}$, it is enough to assume that for all $m \in \Mcal$, there exists $B_m' \geq 1$ such that for all $t \in [n]$,
\begin{equation}
\label{eq_HMM_hyp_tails_modeles}
\forall y \geq 1 \quad \Pbb\left( \sup_{\theta \in \Theta_m} \log \bar\nu^m_\theta(X_t) > B_m' y \, | \, \Fcal_{t-1} \right) \leq e^{-y},
\end{equation}
and then, if~\eqref{eq_HMM_hyp_tails_star} holds, Assumption~\ref{hyp_tails} is satisfied for $B_m = 2 (B^\star \vee B_m' \vee \log(C_Q \log n) \vee (\alpha \log n))$.

Assumption~\ref{densites_continues} holds as soon as the emission densities are continuous.

Assumptions~\ref{hyp_lipschitz_bounded} or~\ref{hyp_lipschitz} require some technical work to verify.
They are used in~\cite{aubert:hal-04526484} to ensure that the entropy of the class of log-densities is controlled. A weaker sufficient assumption for hidden Markov models is introduced in~\cite{lehericy2021hmm_mle} (Assumptions \textbf{[Aentropy]} and \textbf{[Agrowth]}) as well as their Section B.2 to see how it relates to the entropy of the class of log-likelihoods. These two assumptions, together with assuming that the mapping $\theta \in \Theta_m \longmapsto (\pi^m_\theta, Q^m_\theta)$ is Lipschitz, are enough to obtain an oracle inequality. Note that the approach they use in Section B.2 also provides a way to control the regularity of the mappings $\theta \longmapsto \log p^m_\theta$, in the sense that if the parameters of the HMM are Lipschitz and~\eqref{eq_HMM_hyp_tails_star} and~\eqref{eq_HMM_hyp_tails_modeles} hold, then $\theta \longmapsto \log p^m_\theta$ is $\alpha$-Hölder, with $\alpha \gg (\log n)^{-2-\epsilon}$ for any $\epsilon > 0$. A careful reading of our proofs and those of~\cite{aubert:hal-04526484} shows that the oracle inequalities remain valid when assuming the mapping $\alpha$-Hölder instead of Lipschitz in Assumptions~\ref{hyp_lipschitz} and~\ref{hyp_lipschitz_bounded}, up to replacing $D_m$ by $D_m/\alpha$.

Finally, the criterion $\Kbf$ used by~\cite{lehericy2021hmm_mle} is actually the limit of $\Ebb[\Kbf_n]$ when $n \rightarrow +\infty$.
Due to the upper and lower bounds on the entries of the transition matrices, the models forget the past exponentially fast, in the sense that $|\log p^m_\theta(X_t | X_{t-k}^{t-1}) - \log p^m_\theta(X_t | X_{t-k'}^{t-1})| \leq \rho^{(k \wedge k') - 1} / (1-\rho)$ where $\rho = 1 - (C_Q \log n)^{-2}$, see for instance~\cite[Lemma 15]{lehericy2021hmm_mle} or~\cite{douc2004asymptotic}. If a similar property holds for the true process (such as \textbf{[A$\star$forget]} in~\cite{lehericy2021hmm_mle}), then
\begin{equation*}
| \Kbf - \Ebb[\Kbf_n] | = \Ocal((\log n)^4 / n),
\end{equation*}
which is negligible compared to the remainder term of our oracle inequalities. Thus, we may use one or the other loss interchangeably.

\paragraph{Result}

Under the same assumptions as~\cite{lehericy2021hmm_mle}, our results provide an oracle inequality with a penalty that is also of order $(\log n)^\alpha \frac{D_m}{n}$ for some $\alpha > 0$. The main difference is that Assumptions \textbf{[A$\star$mixing]} and \textbf{[A$\star$forgetting]} of \cite{lehericy2021hmm_mle}, a $\rho$-mixing assumption used to obtain concentration inequalities and a forgetting assumption used to truncate the dependencies in the past respectively, are not required to apply our results, provided the loss $\Kbf_n$ is used instead of $\Kbf$.

\subsection{Models of neuronal networks in discrete-time \label{sec:infinite_chain}}

Neurons are electrical cells that communicate via the emission of action potentials, also called spikes \cite{antonio-book}. The shape of the action potential is essentially constant and the important point is the time at which the spikes are emitted. As such, the network is represented by a process $(X_t^i)_{t\in \Zbb, i\in I}$, where $I$ is the set of all neurons constituting the network and $X_t^i=1$ if the neuron $i$ spikes at time $t$ and $X_t^i=0$ otherwise. We consider the filtration $\Fcal_t = \sigma((X^i_s)_{s\leq t,i\in I})$.

Since communication between neurons is not instantaneous, most authors \cite{antonio-book,OstRB} usually assume that conditionally to $\Fcal_{t-1}$, the $(X_t^i)_{i \in I}$ are independent, so that the whole activity can be described by just giving the $p^{\star,i}_t$'s with 
\begin{equation*}
\forall i \in I, \  \forall t \in \Zbb, \quad p^{\star,i}_t = \Pbb(X_t^i=1 \, | \, \Fcal_{t-1}).
\end{equation*}
We assume the process to be stationary.

One of the main neuronal model in discrete-time is the discrete Hawkes process \cite{OstRB}, which can be modeled by 
\begin{equation*}
p^{i,H}_t=\phi_i\left(\sum_{j\in I}\sum_{s<t} h_{j\to i}(t-s)X_s^j\right), 
\end{equation*}
where $\phi_i$ is a rate function that is usually assumed to be Lipschitz, increasing and taking values in $[0,1]$, and where $h_{j\to i}$ are interaction functions: if it is positive at delay $\delta$, neuron $j$ excites neuron $i$ after a delay $\delta$; if it is negative, neuron $j$ inhibits neuron $i$ after a delay $\delta$.
For instance, the linear case is the situation where $\phi_i(x) = \mu_i+x$. In the sequel, to simplify, $\phi_i(.)$ is supposed to be fixed and known. Only the functions $h_{j\to i}$ are unknown.

The Galves-L\"ocherbach neuronal model \cite{antonio-book, galves2013infinite} is slightly different, here
\begin{equation*}
p^{i,GL}_t=\phi_i\left(\sum_{j\in I}\sum_{s=L^i_t}^{t-1} h_{j\to i}(t-s)X_s^j\right), 
\end{equation*}
where $L^i_t$ is the time of the last spike of neuron $i$. In contrast to the Hawkes process, the neurons of this model essentially reset their memory each time they spike.

In practice, only a small finite subset $F$ of $I$ is observed on a finite time duration, say $t=-A,...,n$, for some positive $A$. For a fixed $i\in F$, we are interested by estimating $(p^{\star,i}_t)_{t\in [n]}$ based on the observations of $(X_t^j)_{-A\leq t\leq n, j\in F}.$

We are interested in a specific neuron of interest $i\in F$, so the process $(X_t)_{t \in [n]}$ from our oracle inequalities is taken to be $(X^i_t)_{t=1,...,n}$. The filtration $\Fcal_t$ is the one defined above and generated by the whole network. Finally, in order to define the models, we have access to more information that $(X^i_t)_{t=1,...,n}$ but less than the whole network: we may only use the observations $(X^j_t)_{j\in F, t=-A,...,n}$, which are indeed $\Fcal_t$ adapted.

Whatever the neuronal model ($p^{H}$ for Hawkes or $p^{GL}$ for Galves-L\"ocherbach) that we choose, we need to parameterize it. We define model $m$ by choosing a finite subset of $F$, called $V_m$, which is the proposed neighborhood for neuron $i$ in model $m$, and by choosing $A_m \leq A$ a maximal lag of interaction. In model $m$, all the $h_{j\to i}(u)$ are null if $j \notin V_m$ or $u>A_m$, so the model is parameterized by $(\theta_{j,u}:=h_{j\to i}(u))_{j\in V_m, u=1,...,A_m} \in \Rbb^{A_m|V_m|}$. Given this parameterization, the conditional distributions are defined by
\begin{equation*}
p^{i,m,H}_t=\phi_i\left(\sum_{j\in V_m}\sum_{u=1}^{A_m} \theta_{j,u} X^j_{t-u}\right)
\end{equation*}
and
\begin{equation*}
p^{i,m,GL}_t=\phi_i\left(\sum_{j\in V_m}\sum_{u=1}^{\min(A_m,t-L^i_t)} \theta_{j,u} X^j_{t-u}\right).
\end{equation*}

\paragraph{Validation of the assumptions}

Assumption~\ref{hyp_tails_bounded} means that $p^{\star,i}_t$ as well as $p^{i,m,H}_t$ or $p^{i,m,GL}_t$ are in $[\varepsilon,1-\varepsilon]$. This is a very common assumption in these settings (see \cite{duarte2019estimating,OstRB}). In this sense, Assumption~\ref{hyp_tails} can be seen as a relaxation with respect to previous works.

The assumption that $\phi_i$ is Lipschitz with constant $L$ is a very classical one \cite{duarte2019estimating}. Together with Assumption~\ref{hyp_tails_bounded}, it implies that $\log(\phi_i)$ (probability of a spike) and $\log(1-\phi_i)$ (probability of no spike) are Lipschitz with constant $2\varepsilon^{-1} L$. Thus, the first part of Assumption~\ref{hyp_lipschitz_bounded} is satisfied with $L_m=2\varepsilon^{-1}L$ and the $\ell_1$ norm $\|\theta\|_1$.

For the second part of Assumption~\ref{hyp_lipschitz_bounded}, it depends on $\phi_i$. Indeed, since $\phi_i$ is increasing, we can define
\begin{equation*}
\Theta_m=\left\{\theta \in \Rbb^{A_m|V_m|}\mbox{ such that } \varepsilon\leq \phi_i\left(\sum_{j,u} \theta_{j,u} {\bf 1}_{\theta_{j,u}<0}\right) \mbox{ and }\phi_i\left(\sum_{j,u} \theta_{j,u} {\bf 1}_{\theta_{j,u}>0}\right)\leq 1-\varepsilon \right\},
\end{equation*}
to ensure that Assumption~\ref{hyp_tails_bounded} is satisfied. Since $\phi_i$ is increasing and Lipschitz, we can define its inverse, so that if $\varepsilon$ and $1-\varepsilon$ are possible values for $\phi_i$ (as is typically the case for linear or sigmoid functions) then it automatically follows that for all $\theta\in \Theta_m$
\begin{equation*}
\|\theta\|_1\leq |\phi_i^{-1}(\varepsilon)|+|\phi_i^{-1}(1-\varepsilon)|,
\end{equation*}
and so the second part of Assumption~\ref{hyp_lipschitz_bounded} is satisfied with $M_m= |\phi_i^{-1}(\varepsilon)|+|\phi_i^{-1}(1-\varepsilon)|$.

Assumption~\ref{densites_continues} is automatically fulfilled in each of the models because $p^{i,m,H}_t$ and $p^{i,m,GL}_t$ only depend on a finite set of $X^j_s$. Assumption~\ref{densites_continues} for $p^\star$ can be solved by assuming the following continuity assumption, which is standard in this setting (see \cite{duarte2019estimating,antonio-book}). Let $x$ be a past configuration, i.e. a possible value for $(X^j_s)_{j\in I, s<t}$, and let us remark that by stationarity, $p^{\star,i}_t$ can be seen as a function of $x$ and not of $t$:
\begin{equation*}
p^{\star,i}(x)=\Pbb(X^i_t=1|(X^j_s)_{j\in I,s<t}=x).
\end{equation*}
The continuity assumption of the neuronal model assumes that there exists a nested sequence $(S_k)_{k \geq 1}$ of finite subsets of $I\times \mathbb{Z}_-$ such that $S_k \underset{k \to \infty}{\longrightarrow} I \times \mathbb{Z}$ and such that 
\begin{equation*}
\sup\{|p^{\star,i}_t(x)-p^{\star,i}_t(y)| \text{ such that } x_{|S_k}=y_{|S_k}\} \underset{k \to \infty}{\longrightarrow} 0,
\end{equation*}
where $x_{|S_k}$ is the configuration restricted to the indices in $S_k$.
Informally, this continuity assumption states that one can approximate $p^{\star,i}(x)$ by what happens on a finite number of $x^j_s$'s.
\paragraph{Result}

Following Theorem \ref{generalmodelselectiontheorem_bounded} or its corollary, and since $A_m$ does not depend on $m$, one can take
\begin{equation*}
pen(m)=\Ocal\left( \log(n)^2 \frac{D_m}{n}\right).
\end{equation*}
Note that under mild conditions (see for instance \cite{antonio-book} or \cite{OstRB}), these processes are stationary and $\Ebb(\Kbf_n(p_\theta^m))$ does not depend on $n$. 

Furthermore, since $-\log(p^m_{\theta,t}/p^\star_t)$ is upper and lower bounded, one can easily show that 
\begin{equation}
\label{petittruc}
\Ebb\left(\log^2 \frac{p_t^\star(X_t)}{p_t(X_t)}\, \Big| \Fcal_{t-1}\right) \lesssim \Ebb\left[ \log \frac{p_t^\star(X_t)}{p_t(X_t)} \, \Big| \Fcal_{t-1} \right] \lesssim \Ebb\left(\log^2 \frac{p_t^\star(X_t)}{p_t(X_t)}\, \Big| \Fcal_{t-1} \right),
\end{equation}
where $\lesssim$ means that the inequality holds up to positive multiplicative constant.

Moreover, whatever the neuronal model (Hawkes or Galves-L\"ocherbach), we can expand $\theta^m\in \Theta^m$ with zeroes so that it is defined on $I \times \Nbb^*$. If $\phi_i$ has a derivative that is upper and lower bounded by some positive constant, so do $\log(\phi_i)$ and $\log(1-\phi_i)$, and therefore, whatever the value of $X_t$, 
\begin{equation*}
\left|\sum_{j\in I}\sum_{s<t} (h_{j\to i}(t-s)-\theta^m_{j,t-s})X_{s}^j \right|\lesssim \left|\log\left(\frac{p^{m}_{\theta^m,t}(X_t)}{p^{*}_t(X_t)}\right)\right| \lesssim \left| \sum_{j\in I}\sum_{s<t} (h_{j\to i}(t-s)-\theta^m_{j,t-s})X_{s}^j \right|
\end{equation*}
(for the Hawkes case and with a restricted sum in the lower bound in the Galves-L\"ocherbach case).
Note that both upper and lower bounds are $\Fcal_{t-1}$ measurable.
Hence, going back to the oracle inequality, we can express both the upper bound and the lower bound in terms of the average square distance 
\begin{equation*}
\frac{1}{n}\sum_{t=1}^n\left| \sum_{j\in I}\sum_{s<t} (h_{j\to i}(t-s)-\theta^m_{j,t-s})X_{s}^j \right|^2
\end{equation*}
and the upper bound of the oracle inequality is a trade-off between the bias measured by the average square distance above and the penalty in $D_m/n$ up to logarithmic terms.

Let us compare this result to the ones in \cite{OstRB} for Hawkes and Galves-L\"ocherbach process. In \cite{OstRB}, the authors could only envision linear models (i.e. $\phi_i$ is linear) and were using least-square loss on the $p$. They can perform variable selection thanks to an $\ell_1$ penalization, whereas in our case the summability condition on $\Sigma$ makes it impossible to perform variable selection by considering the full set of subsets of variables. Despite this difference, their oracle inequality is for the exact same average square distance as mentioned above with, in the upper bound, a trade-off between the bias and a term in $D_m/n$ up to logarithmic terms. However, the dimension $D_m$ in their case, had to be smaller than a given a priori level of sparsity $s$. Moreover, their constant in front of $D_m$ was given by an RE inequality on the Gram matrix. In the most general case considered by \cite{OstRB}, this constant depend on the size of the observed network $F$ and explodes with the size of $F$; in the Hawkes case, this bound depends on $s$, and explodes for moderate $s$, whereas in our case the penalty can handle large $D_m$ thanks to the summability condition in $\Sigma$ which ensures that the number of models remains reasonable.

For the Galves-L\"ocherbach model, we can also compare this result with the ones of \cite{duarte2019estimating}. In \cite{duarte2019estimating} (see also \cite{piccioni} for recent improvements), the authors envision the estimation of the interaction neighborhood of a neuron $i$. In these two articles, the authors assume that the set of observed neurons $F$ contains the interaction neighborhood of $i$. In other words, with their assumptions on the process, at least one model is well specified. In this case, our result shows that we can estimate the conditional probability of spiking for $i$ with parametric rate in $D_m/n$ up to logarithmic terms, as the Wilks phenomenon would predict if it was applicable in this case, but without even knowing which model is the true one up front. However, it is not clear if this means that the neighborhood of interaction is correctly estimated by $V_{\hat{m}}$. We cannot manage variable selection per se because the complexity of the family of models would be too large for our general result. But we can at least hope that $V_{\hat{m}}$ would contain the true interaction neighborhood of neuron $i$ with high probability \footnote{This would probably require additional assumptions, such as minimal strength of the interaction inside the neighborhood as in \cite{duarte2019estimating} and \cite{piccioni}.}. On the other hand, our procedure is much less computationally intensive than theirs and does not require the precise examination of patterns of 0 and 1's as they do, which in practice usually prevents them to use it with more than 4 observed neurons \cite{brochini2016estimation}. As such, our method could at least be used to restrict the set of observed neurons before using the methods developed in \cite{duarte2019estimating, piccioni}.


\subsection{Adversarial multi-armed bandits as a cognitive learning model}

All the previous applications were in the stationary case, so that we were able to compare our results with existing ones, even if per se our oracle inequalities do apply even if we do not assume stationarity. Let us now look at a setup which cannot be stationary: learning data, in which we aim to estimate how an individual or system learns to perform a task by observing their training as it takes place. In such a setting, the data that are produced cannot, by essence, be independent or stationary. If in practice many authors have advertised for the use of MLE \cite{daw2011trial, wilson2019ten}, this problem was for the first time studied theoretically in \cite{aubert23}. In this previous work, only the property of the MLE on one model is studied, without the model selection part.

The model used in~\cite{aubert23} is the \texttt{Exp3} algorithm \cite{auer2002nonstochastic}. In the Machine Learning literature, this algorithm was originally proposed to solve an the adversarial multi-armed bandit problem \cite{BubeckBianchi}, \cite{lattimore_szepesvári_2020}, that is a game played sequentially between a learner and an adversary (the environment), where at each round the learner must choose an action $k$ from a set of actions $[K]$ for some integer $K \geq 1$ and the adversary returns a loss $g_{k,t}$ for this action.

\begin{algorithm}[H]
\caption{\texttt{Exp3} with learning rate $\eta$}\label{alg_exp3}
\begin{algorithmic}
\State \textbf{Initialization: } $p_{\eta,1}= \left(\frac{1}{K}, \ldots, \frac{1}{K}\right)$.
\For {$t=1,2,\ldots$}
\State Draw an action $X_t \sim p_{\eta,t}$ and if $X_t=k$ receive a loss $g_{t,k} \in [0,1]$.
\State Update for all $k \in [K]$:
\begin{equation*}
 p_{\eta,t+1}(k) = \frac{\exp \left(-\eta \sum_{s=1}^t \hat{L}_{\eta,s}(k)\right)}{\sum_{j = 1}^K \exp \left(-\eta \sum_{s=1}^t \hat{L}_{\eta,s}(j)\right)} \qquad \text{ where } \hat{L}_{\eta,s}(j) = \frac{g_{j,s}}{p_{\eta,s}(j)} \one_{X_s=j}
\end{equation*}
\EndFor
\end{algorithmic}
\end{algorithm}

The fact that the learner plays as if the environment is adversarial makes it a realistic model for cognitive processes where humans or animals need to adapt to a changing environment. However in a cognitive experiment, most of the time, the environment is not adversarial even if the learner does not know that and still uses an adversarial strategy.

\paragraph{Comparison with \cite{aubert23}}

\cite{aubert23} consider only cases where $g_{k,t}=g_k$ depends only on $k$, which is common in cognitive experiments. It has been shown that if the learning rate $\theta$ is fixed, no estimator can achieve polynomial rates of convergence. This is mainly due to the fact that $p_{\eta,t}(k)$ can rapidly become absurdly small, and as a consequence, only an extremely small number of $X_t$'s are actually relevant to perform the estimation. For instance, if an individual has finished learning and no longer makes mistakes, then it is impossible to improve our estimation of their learning behaviour even with continued observation. Therefore, in \cite{aubert23}, another asymptotic in $T$ is proposed. The goal is to estimate $\theta\in [r,R]$ such that $\eta=\theta/\sqrt{T}$ based on the first
\begin{equation*}
T_\varepsilon=\Big\lfloor \left(\frac{1}{K}-\varepsilon\right)\frac{\sqrt{T}}{R}\Big\rfloor,
\end{equation*}
observations. Indeed up to $T_\epsilon$, all the $p_{\frac{\theta}{\sqrt{T}},t}(k)$ are larger than $\varepsilon$, a fixed constant in $(0,1)$ meant to prevent the vanishing of the various probabilities. The choice of the rate $\sqrt{T}$ can be relaxed to other powers of $T$ but for the present illustrative purpose, we only consider this case, which also corresponds to the renormalization of the learning rate for which \texttt{Exp3} satisfies sublinear regret bounds \cite{BubeckBianchi}.

With this renormalization, \cite{aubert23} proved a convergence in $T_\epsilon^{-1/2}$ for 
\begin{equation*}
\frac{1}{T_\varepsilon}\sum_{t=1}^{T_\varepsilon}\sum_{k=1}^K (p_{\frac{\theta}{\sqrt{T}},t}(k)-p_{\frac{\hat{\theta}}{\sqrt{T}},t}(k))^2,
\end{equation*}
with $\theta$ the true parameter and $\hat{\theta}$ the MLE.

Let us compare this with Theorem \ref{generalmodelselectiontheorem_bounded} or its corollary in the case of a single, well specified, model. Take $\mathcal{F}_t=\sigma(X_1^t)$ and $n=T_\varepsilon$, so that Assumption~\ref{hyp_tails_bounded} is directly satisfied. Since $\log$ is Lipschitz on $[\varepsilon,1-\varepsilon]$, Assumption~\ref{hyp_lipschitz_bounded} follows as a direct consequence of Lemma 4.3 of \cite{aubert23}, with respect to $\|\theta\|_\infty$, with a constant $L$ that only depends on $R$ and $\varepsilon$. Moreover we can always take $M=R$. Finally, Assumption~\ref{densites_continues} is trivial since there is only a finite set of values for $X_t$. 

Thus, Theorem\ref{generalmodelselectiontheorem_bounded} shows that $\Kbf_n(\tilde{p})$ is of order $\log(T_\varepsilon)^2/T_\varepsilon$, which is faster than \cite{aubert23}. Indeed $\Kbf_n(p)$ is comparable to the square norm used in \cite{aubert23}, which can be proved as follows:
\begin{equation*}
\Kbf_n(p)
    = \frac{1}{T_\varepsilon}\sum_{t=1}^\varepsilon \sum_{k=1}^K p^\star_t(k) \log \frac{p^\star_t(k)}{p_t(k)} 
    = \frac{1}{T_\varepsilon}\sum_{t=1}^\varepsilon \sum_{k=1}^K p^\star_t(k) \phi(\log \frac{p_t(k)}{p^\star_t(k)}),
\end{equation*}
with $\phi(u)=e^u-u-1$. Since $|\log \frac{p^\star_t(k)}{p_t(k)}|$ is bounded, $\phi(u)\gtrsim u^2$. This leads to 
\begin{equation*}
\Kbf_n(p) 
    \gtrsim \frac{1}{T_\varepsilon}\sum_{t=1}^\varepsilon \sum_{k=1}^K p^\star_t(k) \log^2 \frac{p^\star_t(k)}{p_t(k)}
    \gtrsim \frac{1}{T_\varepsilon}\sum_{t=1}^\varepsilon \sum_{k=1}^K (p^\star_t(k)-p_t(k))^2,
\end{equation*}
since the derivative of the logarithm is positive and lower bounded on $[\varepsilon, 1-\varepsilon].$
\paragraph{A model selection framework}
Theorem \ref{generalmodelselectiontheorem_bounded} or its corollary goes further than \cite{aubert23} by also allowing for model selection and bias. In particular, this can allow for reward estimation as well as identifying the granularity or precision the learner is able to have in its actions.
This is of particular importance when performing a cognitive learning experiment, since it is rarely clear what numerical value to assign to a real-life reward to reflect how much of an impact it has on the learner (for instance, how much is crab meat worth to an octopus in comparison to shrimp?), nor how precise the learner is able to be when picking an action, in the case where the number of possible actions is continuous or very large (since it is impossible to perfectly replicate a given action; near identical actions should provide near identical rewards, so how dissimilar do actions have to be to be considered distinct?). 

One possibility is to use \texttt{Exp3} with a different parameterization. For a given model $m$, consider a partition $\mathcal{I}^m$ of the set of possible actions (which may even be continuous). Let $D_m$ be the number of disjoint sets in $\mathcal{I}^m$. The parameters of model $m$ are the rewards of each possible set $I$ of the partition, that we call $g^m_I$. Since the role of the learning parameter $\eta$ introduced in Algorithm~\ref{alg_exp3} is redundant with the one of the (unknown) rewards $g$'s, we remove it from the model. 
Thus, under model $m$, the learner proceeds as follows:
\begin{itemize}
\item Initialize $p^m_1=(1/D_m,...,1/D_m)$
\item for $t\geq 1$,
\begin{itemize}
\item pick the set $I_t$ according to $p^m_t$ and pick $X_t$ inside $I_t$ uniformly
\item update for all $I\in\mathcal{I}^m $
\begin{equation*}
p^m_{t+1}(I)= \frac{\exp \left(- \sum_{s=1}^t \hat{L}^m_{s}(I)\right)}{\sum_{J = 1}^{D_m} \exp \left(- \sum_{s=1}^t \hat{L}^m_{s}(J)\right)} \qquad \text{ where } \hat{L}^m_{s}(J) = \frac{g_J^m}{p_{\eta,s}(j)} \one_{X_s\in J}
\end{equation*}
\end{itemize}
\end{itemize}

To obtain a meaningful renormalization, as before, we call $p^m_{\theta^m,t}$ the distribution of $X_t$ given $\mathcal{F}_{t-1}=\sigma(X_1^{t-1})$, when $g_J^m=\frac{\theta_J^m}{\sqrt{T}}$.

As in the previous single model case, Assumption~\ref{hyp_tails_bounded} is straightforward by choosing $n$ of the order of $\sqrt{T}$ up to multiplicative constants. We can also adapt Lemma 4.3 of \cite{aubert23} to show that Assumption~\ref{hyp_lipschitz_bounded} is satisfied as long as all the $\theta_J^m$'s are in $[r,R]$, with a Lipschitz constant $L$ with respect to $\|\theta\|_\infty$ which does not depend on $m$. One can also take $M_m=R$. Assumption~\ref{densites_continues} is again easily satisfied even if the set of actions is continuous.

Therefore, one can choose a penalty of order
\begin{equation*}
pen(m)=\Ocal\left(\log(T_\epsilon)^2\frac{D_m}{T_\varepsilon}\right),
\end{equation*}
and obtain an oracle inequality on $\Kbf_n(\tilde{p})$, which as seen above can be implies an oracle inequality on the square distance. In the end, we obtain a way to estimate at the same time 
\begin{itemize}
    \item with the partition $\mathcal{I}_{\hat{m}}$: the precision in the perception and execution of the actions, that is, which sets of actions are conflated and which are considered distinct, and how precise the learner is able to be when choosing its actions,
    \item with the resulting $\hat{\theta}^{\hat{m}}_J/\sqrt{T_\varepsilon}$: the estimated reward, by which we mean the numerical value that quantifies the average impact of the outcome of the action on the learner's behaviour, modeled by a piecewise constant function on the partition $\mathcal{I}_{\hat{m}}$.
\end{itemize}

\section*{Acknowledgments}

This work was supported by the French government, through the $\text{UCA}^\text{Jedi}$ and 3IA Côte d'Azur Investissements d'Avenir managed by the National Research Agency (ANR-15- IDEX-01 and ANR-19-P3IA-0002),  by the interdisciplinary Institute for Modeling in Neuroscience and Cognition (NeuroMod) of the Université Côte d'Azur and  directly by the National Research Agency ANR-08-JCJC-0125-01, ANR-19-CE40-0024 with the ChaMaNe project, MITI AAP "défi modélisation du vivant" with DYNAMO project and AAP 80PRIME with eXpLAIn. It is part of the \href{https://neuromod.univ-cotedazur.eu/computabrain}{ComputaBrain project}.

\bibliographystyle{apalike}
\bibliography{bibliography}

\begin{thebibliography}{}

\bibitem[Akaike, 1973]{akaike}
Akaike, H. (1973).
\newblock Information theory and an extension of the maximum likelihood
  principle.
\newblock {\em 2nd international symposium on information theory, Akademia
  Kiado, Budapest}, page 267–281.

\bibitem[Arlot, 2019]{arlot}
Arlot, S. (2019).
\newblock {M}inimal penalties and the slope heuristics: a survey.
\newblock {\em J. SFdS}, 160(3):158--168.

\bibitem[Aubert and Leh{\'e}ricy, 2024]{aubert:hal-04526484}
Aubert, J. and Leh{\'e}ricy, L. (2024).
\newblock {Exponential inequalities for suprema of processes with stochastic
  normalization}.
\newblock preprint.

\bibitem[Aubert et~al., 2023]{aubert23}
Aubert, J., Leh\'{e}ricy, L., and Reynaud-Bouret, P. (2023).
\newblock On the convergence of the {MLE} as an estimator of the learning rate
  in the exp3 algorithm.
\newblock In {\em Proceedings of the 40th International Conference on Machine
  Learning}, volume 202 of {\em Proceedings of Machine Learning Research},
  pages 1244--1275. PMLR.

\bibitem[Auer et~al., 2002]{auer2002nonstochastic}
Auer, P., Cesa-Bianchi, N., Freund, Y., and Schapire, R.~E. (2002).
\newblock The nonstochastic multiarmed bandit problem.
\newblock {\em SIAM journal on computing}, 32(1):48--77.

\bibitem[Baraud, 2010]{baraud}
Baraud, Y. (2010).
\newblock {A Bernstein-type inequality for suprema of random processes with
  applications to model selection in non-Gaussian regression}.
\newblock {\em Bernoulli}, 16(4):1064 -- 1085.

\bibitem[Baraud and Birg\'{e}, 2016]{baraud-birge}
Baraud, Y. and Birg\'{e}, L. (2016).
\newblock Rho-estimators for shape restricted density estimation.
\newblock {\em Stochastic Process. Appl.}, 126(12):3888--3912.

\bibitem[Baum and Petrie, 1966]{baum1966hmm}
Baum, L.~E. and Petrie, T. (1966).
\newblock Statistical inference for probabilistic functions of finite state
  markov chains.
\newblock {\em The annals of mathematical statistics}, 37(6):1554--1563.

\bibitem[Bercu et~al., 2015]{bercu2015concentration}
Bercu, B., Delyon, B., Rio, E., et~al. (2015).
\newblock {\em Concentration inequalities for sums and martingales}.
\newblock Springer.

\bibitem[Bercu and Touati, 2008]{bercu2008}
Bercu, B. and Touati, A. (2008).
\newblock Exponential inequalities for self-normalized martingales with
  applications.
\newblock {\em Ann. Appl. Probab.}, 18(5):1848--1869.

\bibitem[Bercu and Touati, 2019]{bercu2019new}
Bercu, B. and Touati, T. (2019).
\newblock {New insights on concentration inequalities for self-normalized
  martingales}.
\newblock {\em Electronic Communications in Probability}, 24(none):1 -- 12.

\bibitem[Bickel and Doksum, 2015]{bickel-doksum}
Bickel, P.~J. and Doksum, K.~A. (2015).
\newblock {\em Mathematical statistics---basic ideas and selected topics.
  {V}ol. 1}.
\newblock Texts in Statistical Science Series. CRC Press, Boca Raton, FL,
  second edition.

\bibitem[Birg\'{e} and Massart, 2001]{birge-massart}
Birg\'{e}, L. and Massart, P. (2001).
\newblock Gaussian model selection.
\newblock {\em J. Eur. Math. Soc. (JEMS)}, 3(3):203--268.

\bibitem[Boucheron and Massart, 2011]{massart-wilk}
Boucheron, S. and Massart, P. (2011).
\newblock A high-dimensional {W}ilks phenomenon.
\newblock {\em Probab. Theory Related Fields}, 150(3-4):405--433.

\bibitem[Brochini et~al., 2016]{brochini2016estimation}
Brochini, L., Galves, A., Hodara, P., Ost, G., and Pouzat, C. (2016).
\newblock Estimation of neuronal interaction graph from spike train data.
\newblock {\em arXiv preprint arXiv:1612.05226}.

\bibitem[Bubeck and Cesa{-}Bianchi, 2012]{BubeckBianchi}
Bubeck, S. and Cesa{-}Bianchi, N. (2012).
\newblock Regret analysis of stochastic and nonstochastic multi-armed bandit
  problems.
\newblock {\em CoRR}, abs/1204.5721.

\bibitem[Burnham and Anderson, 2004]{Burnham}
Burnham, K.~P. and Anderson, D.~R. (2004).
\newblock Multimodel inference: understanding {AIC} and {BIC} in model
  selection.
\newblock {\em Sociol. Methods Res.}, 33(2):261--304.

\bibitem[Castellan, 2003]{castellan}
Castellan, G. (2003).
\newblock Density estimation via exponential model selection.
\newblock {\em IEEE Trans. Inform. Theory}, 49(8):2052--2060.

\bibitem[Cox, 1975]{cox1975partial}
Cox, D.~R. (1975).
\newblock Partial likelihood.
\newblock {\em Biometrika}, 62(2):269--276.

\bibitem[Daw, 2011]{daw2011trial}
Daw, N.~D. (2011).
\newblock Trial-by-trial data analysis using computational models.
\newblock {\em Decision making, affect, and learning: Attention and performance
  XXIII}, 23(1).

\bibitem[De~Castro et~al., 2016]{dCGL2016minimax}
De~Castro, Y., Gassiat, {\'E}., and Lacour, C. (2016).
\newblock Minimax adaptive estimation of nonparametric hidden {M}arkov models.
\newblock {\em Journal of Machine Learning Research}, 17:1--43.

\bibitem[de~la Pe\~na et~al., 2004]{de_la_Pe_a_2004}
de~la Pe\~na, V.~H., Klass, M.~J., and Leung~Lai, T. (2004).
\newblock Self-normalized processes: exponential inequalities, moment bounds
  and iterated logarithm laws.
\newblock {\em The Annals of Probability}, 32(3).

\bibitem[de~la Pe\~{n}a et~al., 2009]{pena-book}
de~la Pe\~{n}a, V.~H., Lai, T.~L., and Shao, Q.-M. (2009).
\newblock {\em Self-normalized processes}.
\newblock Probability and its Applications (New York). Springer-Verlag, Berlin.
\newblock Limit theory and statistical applications.

\bibitem[De~Santis et~al., 2022]{piccioni}
De~Santis, E., Galves, A., Nappo, G., and Piccioni, M. (2022).
\newblock Estimating the interaction graph of stochastic neuronal dynamics by
  observing only pairs of neurons.
\newblock {\em Stochastic Processes and their Applications}, 149:224--247.

\bibitem[Douc et~al., 2004]{douc2004asymptotic}
Douc, R., Moulines, {\'E}., and Ryd{\'e}n, T. (2004).
\newblock Asymptotic properties of the maximum likelihood estimator in
  autoregressive models with {M}arkov regime.
\newblock {\em Annals of Statistics}, pages 2254--2304.

\bibitem[Duarte et~al., 2019]{duarte2019estimating}
Duarte, A., Galves, A., L{\"o}cherbach, E., and Ost, G. (2019).
\newblock Estimating the interaction graph of stochastic neural dynamics.
\newblock {\em Bernoulli}, 25(1):771--792.

\bibitem[Galves and L{\"o}cherbach, 2013]{galves2013infinite}
Galves, A. and L{\"o}cherbach, E. (2013).
\newblock Infinite systems of interacting chains with memory of variable
  length-a stochastic model for biological neural nets.
\newblock {\em Journal of Statistical Physics}, 151(5).

\bibitem[Galves et~al., 2024]{antonio-book}
Galves, A., L{\"o}cherbach, E., and Pouzat, C. (2024).
\newblock {\em Probabilistic Spiking Neuronal Nets – Data, Models and
  Theorems}.
\newblock Springer.

\bibitem[Ghosal and van~der Vaart, 2007]{ghosal2007posterior}
Ghosal, S. and van~der Vaart, A. (2007).
\newblock Posterior convergence rates of dirichlet mixtures at smooth
  densities.
\newblock {\em Ann. Statist.}, 35(1):697--723.

\bibitem[Hansen et~al., 2015]{Hansen_2015}
Hansen, N.~R., Reynaud-Bouret, P., and Rivoirard, V. (2015).
\newblock Lasso and probabilistic inequalities for multivariate point
  processes.
\newblock {\em Bernoulli}, 21(1).

\bibitem[Lacour, 2008a]{lacourHMM}
Lacour, C. (2008a).
\newblock Adaptive estimation of the transition density of a particular hidden
  {M}arkov chain.
\newblock {\em J. Multivariate Anal.}, 99(5):787--814.

\bibitem[Lacour, 2008b]{lacourMarkov}
Lacour, C. (2008b).
\newblock Nonparametric estimation of the stationary density and the transition
  density of a {M}arkov chain.
\newblock {\em Stochastic Process. Appl.}, 118(2):232--260.

\bibitem[Lattimore and Szepesvári, 2020]{lattimore_szepesvári_2020}
Lattimore, T. and Szepesvári, C. (2020).
\newblock {\em Bandit Algorithms}.
\newblock Cambridge University Press.

\bibitem[Leh{\'e}ricy, 2021]{lehericy2021hmm_mle}
Leh{\'e}ricy, L. (2021).
\newblock Nonasymptotic control of the mle for misspecified nonparametric
  hidden markov models.
\newblock {\em Electronic Journal of Statistics}, 15(2):4916--4965.

\bibitem[Lehmann, 1999]{Lehmann-book}
Lehmann, E.~L. (1999).
\newblock {\em Elements of large-sample theory}.
\newblock Springer Texts in Statistics. Springer-Verlag, New York.

\bibitem[Massart, 2007]{massart2007concentration}
Massart, P. (2007).
\newblock {\em Concentration inequalities and model selection: Ecole d'Et{\'e}
  de Probabilit{\'e}s de Saint-Flour XXXIII-2003}.
\newblock Springer.

\bibitem[Ost and Reynaud-Bouret, 2020]{OstRB}
Ost, G. and Reynaud-Bouret, P. (2020).
\newblock Sparse space-time models: concentration inequalities and {L}asso.
\newblock {\em Ann. Inst. Henri Poincar\'{e} Probab. Stat.}, 56(4):2377--2405.

\bibitem[Ramponi et~al., 2020]{ramponi2020inverse}
Ramponi, G., Drappo, G., and Restelli, M. (2020).
\newblock Inverse reinforcement learning from a gradient-based learner.
\newblock {\em Advances in Neural Information Processing Systems},
  33:2458--2468.

\bibitem[Shen et~al., 2013]{shen2013adaptive}
Shen, W., Tokdar, S.~T., and Ghosal, S. (2013).
\newblock Adaptive bayesian multivariate density estimation with dirichlet
  mixtures.
\newblock {\em Biometrika}, 100(3):623--640.

\bibitem[Spokoiny, 2012]{spok2012}
Spokoiny, V. (2012).
\newblock Parametric estimation. {F}inite sample theory.
\newblock {\em Ann. Statist.}, 40(6):2877--2909.

\bibitem[Spokoiny, 2017]{spok2017}
Spokoiny, V. (2017).
\newblock Penalized maximum likelihood estimation and effective dimension.
\newblock {\em Ann. Inst. Henri Poincar\'{e} Probab. Stat.}, 53(1):389--429.

\bibitem[Sur and Cand\`es, 2019]{Candes2019-court}
Sur, P. and Cand\`es, E.~J. (2019).
\newblock A modern maximum-likelihood theory for high-dimensional logistic
  regression.
\newblock {\em Proc. Natl. Acad. Sci. USA}, 116(29):14516--14525.

\bibitem[Sur et~al., 2019]{Candes2019}
Sur, P., Chen, Y., and Cand\`es, E.~J. (2019).
\newblock The likelihood ratio test in high-dimensional logistic regression is
  asymptotically a {\it rescaled} chi-square.
\newblock {\em Probab. Theory Related Fields}, 175(1-2):487--558.

\bibitem[Talagrand, 1996]{talagrand}
Talagrand, M. (1996).
\newblock New concentration inequalities in product spaces.
\newblock {\em Invent. Math.}, 126(3):505--563.

\bibitem[Tang and Leng, 2010]{tang}
Tang, C.~Y. and Leng, C. (2010).
\newblock Penalized high-dimensional empirical likelihood.
\newblock {\em Biometrika}, 97(4):905--919.
\newblock With supplementary material available online.

\bibitem[van~de Geer, 1995]{vandeGeer}
van~de Geer, S. (1995).
\newblock Exponential inequalities for martingales, with application to maximum
  likelihood estimation for counting processes.
\newblock {\em Ann. Statist.}, 23(5):1779--1801.

\bibitem[van~der Vaart et~al., 1996]{vdw1996Mestimators}
van~der Vaart, A.~W., Wellner, J.~A., van~der Vaart, A.~W., and Wellner, J.~A.
  (1996).
\newblock {M}-estimators.
\newblock {\em Weak Convergence and Empirical Processes: With Applications to
  Statistics}, pages 284--308.

\bibitem[Wilks, 1938]{wilks}
Wilks, S. (1938).
\newblock The large-sample distribution of the likelihood ratio for composite
  testing hypotheses.
\newblock {\em Ann. Math. Stat.}, pages 60--62.

\bibitem[Wilson and Collins, 2019]{wilson2019ten}
Wilson, R.~C. and Collins, A.~G. (2019).
\newblock Ten simple rules for the computational modeling of behavioral data.
\newblock {\em Elife}, 8:e49547.

\bibitem[Zheng et~al., 2019]{zheng}
Zheng, C., Ferrari, D., and Yang, Y. (2019).
\newblock Model selection confidence sets by likelihood ratio testing.
\newblock {\em Statist. Sinica}, 29(2):827--851.

\end{thebibliography}

\appendix

\section{Proof of Theorems~\ref{generalmodelselectiontheorem_bounded} and~\ref{generalmodelselectiontheorem}}\label{sectionproofgeneralmodelselection}

Both theorems share the same proof, up to differences in notations due to picking Assumptions~\ref{hyp_tails_bounded} and~\ref{hyp_lipschitz_bounded} or Assumptions~\ref{hyp_tails} and~\ref{hyp_lipschitz}, which will be introduced when required.


In the sequel to avoid losing track of important dependance, we denote: for any $t \in [n]$, $m \in \Mcal$, $\theta \in \Theta_m$, $x_t \mapsto p^m_{\theta,t}(x_t |\Fcal_{t-1})$ the density of $X_t$ conditionally to $\Fcal_{t-1}$ under the parameter $\theta$, and likewise for $p^\star_t(x_t |\Fcal_{t-1} )$.

In what follows, fix $m \in \Mcal$ and $\Bar{\theta}^m \in \Theta_m$. Consider the following functions, defined for all $t \in [n]$, $x_1^t \in \Xcal^t$, $m' \in \Mcal$ and $\delta \in \Theta_{m'}$ by
\begin{equation*}
g_{\delta,t}^{m'}(x_t, \Fcal_{t-1})
 = - \log\left(\frac{p_{\delta,t}^{m'}(x_t | \Fcal_{t-1})}{p^\star_t(x_t | \Fcal_{t-1})}\right)
\end{equation*}
and write $g_\delta^{m'} = (g_{\delta,t}^{m'})_{t \in [n]}$.

For all $m' \in \Mcal$, let $\hat{\theta}^{m'}$ be a maximizer of $\theta \in \Theta_{m'} \mapsto \frac{1}{n} \ell_n(\theta)$, and let $\crit(m')$ be
\begin{equation*}
 \crit(m') = - \frac{1}{n} \ell_n(\hat{\theta}^{m'}) + \pen(m'),
\end{equation*}
and define the set $\Mcal'$ as
\begin{equation*}
 \Mcal' = \{m' \in \Mcal, \quad \crit(m') \leq \crit(m)\}.
\end{equation*}
For any family $h = (h_t)_{t \in [n]}$ of functions $X_t$ that may depend on the past, that is $h_t(X_t,\Fcal_{t-1})$, we write 
\begin{equation*}
    \left\{
    \begin{aligned}
        & P(h) = \frac{1}{n} \sum_{t=1}^n h_t(X_t, \Fcal_{t-1}), \\
        & C(h) = \frac{1}{n} \sum_{t=1}^n \Ebb[h_t(X_t, \Fcal_{t-1}) \, | \, \Fcal_{t-1}] 
        \quad \text{the compensator of } P(h), \\
        & \nu(h) = P(h) - C(h)
        = \frac{1}{n} \sum_{t=1}^n \left(h_t(X_t, \Fcal_{t-1}) - \Ebb[h_t(X_t, \Fcal_{t-1}) \, | \,\Fcal_{t-1}] \right). 
    \end{aligned}
    \right.
\end{equation*}
By definition, for every $m' \in \Mcal'$,
\begin{equation*}
P(g_{\hat{\theta}^{m'}}^{m'}) + \pen(m')
 \leq P(g_{\hat{\theta}^{m}}^{m}) + \pen(m)
 \leq P(g_{\Bar{\theta}^{m}}^{m}) + \pen(m).
\end{equation*}
Therefore, since $P = C + \nu$, for every $m' \in \Mcal'$,
\begin{equation*}
C\left(g_{\hat{\theta}^{m'}}^{m'}\right) + \nu(g_{\hat{\theta}^{m'}}^{m'}) 
 \leq C\left( g_{\Bar{\theta}^m}^m\right) + \pen(m) + \nu(g_{\Bar{\theta}^m}^m) - \pen(m').
\end{equation*}
Plugging the definition of $g_{\Bar{\theta}^m}^m$ in the above leads to
\begin{equation}\label{massartdebut}
\Kbf_n(p_{\hat{\theta}^{m'}}^{m'})
 \leq \Kbf_n(p_{\Bar{\theta}^m}^{m})
 + \pen(m) - \nu(g_{\hat{\theta}^{m'}}^{m'}) - \pen(m')+\nu(g_{\Bar{\theta}^m}^m).
\end{equation}

So far everything is similar to \cite{massart2007concentration}. The goal is now to control $-\nu(g_{\hat{\theta}^{m'}}^{m'})$.

Let us distinguish between the two sets of assumptions. Since the core of the proof is identical in both cases, we introduce notations to encompass both cases in a single setting.
\begin{itemize}
\item Under Assumptions~\ref{hyp_tails_bounded} and~\ref{hyp_lipschitz_bounded}, for each $t \in [n]$ and $m'' \in \Mcal$, let $q_n = 0$, $F^\infty_{m''} = 2 \log (\varepsilon^{-1})$ and $F^{\text{lip}}_{m''} = L_{m''}$, and for any $\delta \in \Theta_{m''}$, let
\begin{equation*}
\Vbf_n(p^{m''}_\delta)
 = \frac{1}{n} \sum_{t=1}^n \Ebb\left[
 \left( \log 
 \frac{p_t^\star(X_t | \Fcal_{t-1})}{p^{m''}_{\delta,t}(X_t | \Fcal_{t-1})} \right)^2
 \, \Big| \,\Fcal_{t-1} \right].
\end{equation*}

\item Under Assumptions~\ref{hyp_tails} and~\ref{hyp_lipschitz}, for each $t \in [n]$ and $m'' \in \Mcal$, let $q_n = 2n^{-1}$, $F^\infty_{m''} = B_{m''} \log n$ and $F^{\text{lip}}_{m''} = L_{m''} \log n$, and for any $\delta \in \Theta_{m''}$, let
\begin{equation*}
\Vbf_n(p^{m''}_\delta)
    = \frac{1}{n} \sum_{t=1}^n \Ebb\left[
        \left( \log 
        \frac{p_t^\star(X_t | \Fcal_{t-1})}{p^{m''}_{\delta,t}(X_t |\Fcal_{t-1})} \right)^2
        \one_{\left|\log 
        \frac{p_t^\star(X_t | \Fcal_{t-1})}{p^{m''}_{\delta,t}(X_t |\Fcal_{t-1})}
        \right| \leq F^\infty_{m''}}
     \, \Big| \, 
     \Fcal_{t-1}
 \right].
\end{equation*}
\end{itemize}

Let
\begin{equation}
\label{eq_def_vw}
A_{m'} = F^{\text{lip}}_{m'} M_{m'} + F^\infty_{m'}
 \quad \text{and} \quad
 v_{m'} = A_{m'} \sqrt{2n}.
\end{equation}

For any $m' \in \Mcal$, let $\sigma_{m'}$ be the solution of the equation
\begin{equation}\label{defsigmam}
\sigma
 = \left(1 \wedge \frac{v_{m'}}{\sigma} \right) \sqrt{(D_{m'}+1) \log \left(\frac{v_{m'}}{\sigma} \vee e \right)}
 + \frac{A_{m'}}{\sigma} (D_{m'}+1) \log \left(\frac{v_{m'}}{\sigma} \vee e \right)
\end{equation}

\begin{lemma}
\label{controlofdiffnufm}
Assume Assumption~\ref{densites_continues} holds, as well as either Assumptions~\ref{hyp_tails_bounded} and~\ref{hyp_lipschitz_bounded} or Assumptions~\ref{hyp_tails} and~\ref{hyp_lipschitz}.
For any family $(\eta_{m'})_{m' \in \Mcal}$ taking values in $(0,1)$, letting $y_{m'} = \eta_{m'}^{-1} \sqrt{\sigma_{m'}^2 + x+D_{m'}}$ for each $m' \in \Mcal$, it holds with probability at least $1- q_n - \displaystyle 6 e^{-x} \sum_{m' \in \Mcal} \log(v_{m'}) e^{-D_{m'}}$ that for all $m' \in \Mcal$,
\begin{equation}\label{eqeta2}
\sup_{\delta \in \Theta_{m'}} \left(\frac{|\nu(g_{\delta}^{m'})|}{2\Vbf_n(p_{\delta}^{m'}) + \frac{1}{n} y_{m'}^2}\right)
 \leq 80 (2 \eta_{m'} + \eta_{m'}^2 A_{m'}).
\end{equation}
\end{lemma}

\begin{proof}
 See Section \ref{proofofdiffnufm}.
\end{proof}

\begin{lemma}
\label{lemmaVariance}
Under either Assumptions~\ref{hyp_tails_bounded} and~\ref{hyp_lipschitz_bounded} or Assumptions~\ref{hyp_tails} and~\ref{hyp_lipschitz},
almost surely, for all $m' \in \Mcal$ and $\delta \in \Theta_{m'}$,
\begin{equation*}
\Vbf_n(p_\delta^{m'})
 \leq 16 (F^\infty_{m'})^2 \Kbf_n(p_\delta^{m'}).
\end{equation*}
\end{lemma}

\begin{proof}
 See Section~\ref{section_proofoflemmaVariance}.
\end{proof}

Fix some sequence $(\eta_{m'})_{m' \in \Mcal}$ in $(0,1)$ to be determined later and let $y_{m'} = \eta_{m'}^{-1} \sqrt{\sigma_{m'}^2 + x+D_{m'}}$ for all $m' \in \Mcal$.
By Lemma~\ref{controlofdiffnufm}, and using the definition of $\Sigma$ and the fact that since $A_{m'} \geq 2$, $n \geq 2$ and $x \geq 0$,
\begin{equation*}
\log(v_{m'})
 = \log A_{m'} + \frac{1}{2} \log (2n)
 \leq 3 \log(n) \log A_{m'},
\end{equation*}
it holds with probability at least $1 - q_n - 18 \log(n) \Sigma e^{-x}$, for all $m' \in \Mcal'$,
\begin{equation*}
- \nu(g_{\hat{\theta}^{m'}}^{m'})
 \leq 80 (2 \eta_{m'} + \eta_{m'}^2 A_{m'}) \left( 32 (F^\infty_{m'})^2 \Kbf_n(p_{\hat{\theta}^{m'}}^{m'}) + \frac{1}{n} y_{m'}^2 \right)
\end{equation*}
and
\begin{equation*}
\nu(g_{\Bar{\theta}^m}^m)
 \leq 80 (2 \eta_m + \eta_m^2 A_m) \left( 32 (F^\infty_m)^2 \Kbf_n(p_{\Bar{\theta}^m}^m) + \frac{1}{n} y_m^2 \right).
\end{equation*}
Injecting this result in~\eqref{massartdebut}, with probability at least $1 - q_n - 18 \log(n) \Sigma e^{-x}$, for all $m' \in \Mcal'$,
\begin{multline}
\label{modelselectionproba}
\left( 1 - 2560 (F^\infty_{m'})^2 (2 \eta_{m'} + \eta_{m'}^2 A_{m'}) \right) \Kbf_n(p_{\hat{\theta}^{m'}}^{m'}) \\
 \leq \left( 1 + 2560 (F^\infty_m)^2 (2 \eta_m + \eta_m^2 A_m) \right) \Kbf_n(p_{\Bar{\theta}^m}^{m}) \\
 + 80 \left(\frac{2}{\eta_m} + A_m\right) \frac{1}{n} (\sigma_m^2 + x + D_m) + \pen(m) \\
 + 80 \left(\frac{2}{\eta_{m'}} + A_{m'}\right) \frac{1}{n} (\sigma_{m'}^2 + x + D_{m'}) - \pen(m').
\end{multline}
By the definition of $\sigma_{m'}$ in~\eqref{defsigmam}, it is smaller than the solution $\sigma_{m'}'$ of
\begin{equation}
\sigma
 = \displaystyle \sqrt{(D_{m'} + 1) \log \left(\frac{v_{m'}}{\sigma} \vee e \right)}
 + \frac{A_{m'}}{\sigma} (D_{m'} + 1) \log \left(\frac{v_{m'}}{\sigma} \vee e \right),
\end{equation}
which satisfies
\begin{align*}
\sqrt{A_{m'} (D_{m'} + 1)}
 \leq \sigma_{m'}'
 &\leq \sqrt{(D_{m'} + 1) \log\left(\frac{\sqrt{2n A_{m'}}}{\sqrt{D_{m'} + 1}} \vee e \right)} \\
 & \qquad + \sqrt{A_{m'} (D_{m'} + 1)} \log\left(\frac{\sqrt{2n A_{m'}}}{\sqrt{D_{m'} + 1}} \vee e \right) \\
 &\leq 2 \sqrt{A_{m'} (D_{m'} + 1)} \log((n A_{m'}) \vee e).
\end{align*}

Note that $n A_{m'} \geq e$.
Let $\kappa \in (0,1]$ and for all $m' \in \Mcal$, let $\eta_{m'} = \frac{c \kappa}{A_{m'} (F^\infty_{m'})^{3/2}}$ for some small enough numerical constant $c > 0$, and recall that $A_{m'} \geq F^\infty_{m'}$, then there exist numerical constants $C_{\pen}$ and $C$ such that if for all $m' \in \Mcal'$,
\begin{equation}
\label{choixpenalite}
 \pen(m') \geq \frac{C_{\pen}}{\kappa} A_{m'}^2 (F^\infty_{m'})^{3/2} \log(nA_{m'})^2 \frac{D_{m'}}{n},
\end{equation}
it holds for all $x \geq 0$, with probability at least $1 - q_n - 18 \log (n) \Sigma e^{-x}$, for all $m' \in \Mcal'$,
\begin{multline}\label{eq_concentration_proba}
( 1 - \kappa ) \Kbf_n(p_{\hat{\theta}^{m'}}^{m'})
 \leq ( 1 + \kappa ) \Kbf_n(p_{\Bar{\theta}^m}^{m})
 + 2\pen(m) \\
 + \frac{C}{\kappa} (A_{m'} (F^\infty_{m'})^{3/2} \log(nA_{m'})^2 + A_m (F^\infty_m)^{3/2} \log(nA_m)^2) \frac{x}{n}.
\end{multline}

\section{Proof of Corollaries~\ref{cor_generalmodelselection_bounded} and~\ref{cor_generalmodelselection}}
\label{sec_proof_corollaries}

Corollary~\ref{cor_generalmodelselection_bounded} follows directly from the fact that $\Ebb[Z] \leq \int_{t \geq 0} \Pbb(Z \geq t) dt$ for any random variable $Z$.

Under Assumptions~\ref{hyp_tails}, let us first show that $\Kbf_n(p_{\delta}^{m'})$ is bounded for all $\delta$ and $m'$ almost surely.
By Assumption~\ref{hyp_tails}, almost surely, for any $m' \in \Mcal$ and $\delta \in \Theta_{m'}$,
\begin{align*}
\Kbf_n(p_{\delta}^{m'})
 &\leq \frac{1}{n} \sum_{t=1}^n \Ebb\left[ \left|\log 
 \frac{p_t^\star(X_t | \Fcal_{t-1})}{p_{\delta,t}^{m'}(X_t | \Fcal_{t-1})} \right| \, \Big| \, \Fcal_{t-1} \right]\\
 &\leq \frac{1}{n} \sum_{t=1}^n \left( B_{m'} + B_{m'} \int_{1}^{+\infty} \Pbb\left(\left|\log 
 \frac{p_t^\star(X_t | \Fcal_{t-1})}{p_{\delta,t}^{m'}(X_t | \Fcal_{t-1})} \right| \geq B_{m'} y \, \Big| \, \Fcal_{t-1} \right)dy \right) \\
 &\leq 2 B_{m'}.
\end{align*}

To conclude, assume that there exist $A,B > 0$ such that $\sup_{m \in \Mcal} B_m \leq B(n)$ and $\sup_{m \in \Mcal} A_m \leq A(n)$, so that by Theorem~\ref{generalmodelselectiontheorem}, with probability at least $1 - 2n^{-1} - 18 \log(n) \Sigma e^{-x}$,
\begin{multline*}
(1 - \kappa) \Kbf_n(\tilde{p})
 \leq \inf_{m \in \Mcal} \left( (1 + \kappa) \inf_{\theta \in \Theta^{D_m}} \Kbf_n(p_{\theta}^{m})
 + 2 \pen(m)
 \right) \\
 + \frac{2 C'}{\kappa} A(n) B(n)^{3/2} \log(nA(n))^2 \frac{(\log n)^{5/2} x}{n},
\end{multline*}
and use that for any random variable $Z$ such that $Z \leq M$ a.s. for some constant $M > 0$, $\Ebb[Z] \leq \int_{t = 0}^M \Pbb(Z \geq t) dt$, so that for all $\kappa \in (0,1]$,
\begin{multline*}
(1 - \kappa) \Ebb \left[\Kbf_n(p_{\hat{\theta}^{\hat{m}}}^{\hat{m}})\right]
 \leq \Ebb\left[ \inf_{m \in \Mcal} \left( (1 + \kappa) \inf_{\theta \in \Theta^{D_m}} \Kbf_n(p_{\theta}^{m}) + 2 \pen(m)
 \right) \right] \\
 + \frac{4 B(n)}{n}
 + \frac{36 C'}{\kappa} \Sigma A(n) B(n)^{3/2} \log(nA(n))^2 \frac{(\log n)^{7/2}}{n},
\end{multline*}
and the last term dominates the second to last one.

\section{Proof of the Lemmas}

\subsection{Proof of Lemma~\ref{controlofdiffnufm}}
\label{proofofdiffnufm}

Fix $m' \in \Mcal$.
Recall that for any $\delta, \eta \in \Theta_{m'}$,
\begin{equation*}
\nu(g_{\delta}^{m'}) - \nu(g_{\eta}^{m'})
 = \frac{1}{n} \sum_{t=1}^n \int
 \log\left(
 \frac{p_{\eta,t}^{m'}(x_t | \Fcal_{t-1})}{p_{\delta,t}^{m'}(x_t | \Fcal_{t-1})} \right) (d\delta_{X_t}(x_t) - p^\star_t(x_t | \Fcal_{t-1})d\dommes(x_t)) ,
\end{equation*}
where $\delta_a$ is the Dirac measure in $a$.

We extend $\Theta_{m'}$ into $\Theta_{m'} \cup \{\star\}$ by defining $p^{m'}_\star = p^\star$, so that when $\eta = \star$, $\nu(g_{\eta}^{m'}) = 0$ and the formula above becomes $\nu(g_{\delta}^{m'})$.
We want to control this uniformly over $\delta, \eta \in \Theta_{m'} \cup \{\star\}$.

Fix $\delta, \eta \in \Theta_{m'} \cup \{\star\}$.
For any $t \in [n]$, let
\begin{equation*}
 \Delta_{t} = \int
 \log\left(
 \frac{p_{\eta,t}^{m'}(x_t | \Fcal_{t-1})}{p_{\delta,t}^{m'}(x_t | \Fcal_{t-1})} \right) (d\delta_{X_t}(x_t) - p^\star_t(x_t |\Fcal_{t-1})d\dommes(x_t)).
\end{equation*}
For any $t \in [n+1]$, let $M_t = \sum_{s=1}^{t-1} \Delta_s$ (in particular, $M_1 = 0$), so that $\frac{1}{n} M_{n+1} = \nu(g_{\delta}^{m'}) - \nu(g_{\eta}^{m'})$.

$(M_{t})_{t \geq 1}$ is a $(\sigma(X_1^{t-1}))_{t \geq 1}$-martingale. For $\ell \geq 2$, let $C_1^{\ell}= 0$ and for $t \geq 2$, let
\begin{equation*}
 C_t^{\ell} = \sum_{s=1}^{t-1} \Ebb [ \Delta_s^\ell \, | \, 
 \Fcal_{t-1}].
\end{equation*}

Note that for all $s \in [n]$,
\begin{equation*}
 | \Delta_s | \leq 2 \int
 \left| \log 
 \frac{p_{\eta,s}^{m'}(x_s | \Fcal_{s-1})}{p_{\delta,s}^{m'}(x_s | \Fcal_{s-1})} \right| \frac{d\delta_{X_s}(x_s) + p^\star_s(x_s | \Fcal_{s-1})d\dommes(x_s)}{2},
\end{equation*}
so that by convexity of $x \mapsto x^\ell$,
\begin{align}
\nonumber
|C_t^{\ell}|
 &\leq \sum_{s=1}^{t-1} \Ebb \left[
 2^\ell \int \left| \log 
 \frac{p_{\eta,s}^{m'}(x_s | \Fcal_{s-1})}{p_{\delta,s}^{m'}(x_s | \Fcal_{s-1})} \right|^\ell \frac{d\delta_{X_s}(x_s) + p^\star_s(x_s |\Fcal_{s-1})d\dommes(x_s)}{2}
 \, \bigg| \, \Fcal_{s-1} \right] \\
\label{partialAt}
 & = 2^\ell \sum_{s=1}^{t-1} \int \left| \log 
 \frac{p_{\eta,s}^{m'}(x_s | \Fcal_{s-1})}{p_{\delta,s}^{m'}(x_s | \Fcal_{s-1})} \right|^\ell p^\star_s(x_s | \Fcal_{s-1})d\dommes(x_s).
\end{align}

Let $\Qcal$ be a countable dense subset of $\Xcal$.

Now, let us distinguish two cases:

\begin{itemize}
\item Under Assumptions~\ref{hyp_tails_bounded} and~\ref{hyp_lipschitz_bounded}, for each $t \in [n]$, let $A_n^t$ be the event defined by 
\begin{multline*}
 \Bigg\{
 \forall s \in [t], \ \forall m'' \in \Mcal, \ \forall \delta, \theta \in \Theta_{m''} \cup \{\star\}, \ \left| \log \frac{p^{m''}_\delta(X_s | \Fcal_{s-1})}{p^{m''}_\theta(X_s | \Fcal_{s-1})} \right| \leq 2 \log \frac{1}{\varepsilon} \\
 \text{and} \quad
 \forall s \in [t], \ \forall {m''} \in \Mcal, \ \forall \delta, \theta \in \Theta_{m''}, \ \left| \log \frac{p^{m''}_\delta(X_s | \Fcal_{s-1})}{p^{m''}_\theta(X_s | \Fcal_{s-1})} \right| \leq L_{m''} \|\delta - \theta\|_{m''}
\Bigg\}.
\end{multline*}
 Thanks to Assumptions~\ref{hyp_tails_bounded} and~\ref{hyp_lipschitz_bounded}, $\Pbb(A_n^t)=1$. Write $F^\infty_{m'} = 2 \log(\varepsilon^{-1})$ and $F^{\text{lip}}_{m'} = L_{m''}$.

\item Under Assumptions~\ref{hyp_tails} and~\ref{hyp_lipschitz}, for each $t \in [n]$, let $A_n^t$ be the event defined by 
\begin{multline*}
\Bigg\{
 \forall s \in [t], \ \forall m'' \in \Mcal, \ \forall \delta, \theta \in \Theta_{m''} \cup \{\star\}, \ \left| \log \frac{p^{m''}_\delta(X_s | \Fcal_{s-1})}{p^{m''}_\theta(X_s | \Fcal_{s-1})} \right| \leq 2 B_{m''} \log n \\
 \text{and} \quad
 \forall s \in [t], \ \forall {m''} \in \Mcal, \ \forall \delta, \theta \in \Theta_{m''}, \ \left| \log \frac{p^{m''}_\delta(X_s | \Fcal_{s-1})}{p^{m''}_\theta(X_s | \Fcal_{s-1})} \right| \leq L_{m''} \|\delta - \theta\|_{m''} \log n
\Bigg\}
\end{multline*}
Thanks to Assumptions~\ref{hyp_tails} and~\ref{hyp_lipschitz}, $\Pbb(A_n^t) \geq 1 - 2n^{-1}$. Write $F^\infty_{m'} = 2 B_{m''} \log n$ and $F^{\text{lip}}_{m'} = L_{m''} \log n$.
\end{itemize}

%
%

Note that in both cases, $A_n^t$ depends on $X_t$ and $\Fcal_{t-1}$. So given $\Fcal_{t-1}$, we can ask ourselves what are the possible values $x$ of $X_t$, so that $A_n^t $ is true. To simplify notation, we write $x \in A_n^t$ when this is true. The event $\{x\in A_n^t\}$ is therefore $\Fcal_{t-1}$ measurable.

Now, under Assumption~\ref{densites_continues}, we may define
\begin{align*}
R_{\infty,n}(\delta, \eta)
 &= \max_{1 \leq s \leq n}
 \sup_{x \in \Qcal} \left( \left| \log 
 \frac{p_{\eta,s}^{m'}(x | \Fcal_{s-1})}{p_{\delta,s}^{m'}(x |\Fcal_{s-1})} \right| \one_{x \in A_n^s} \right)
\end{align*}
and
\begin{align*}
R_{2,n}(\delta, \eta)^2
 &= 2 \sum_{s=1}^{n} \Ebb\left[ \left( \log 
 \frac{p_{\eta,s}^{m'}(X_s | \Fcal_{s-1})}{p_{\delta,s}^{m'}(X_s | \Fcal_{s-1})} \right)^2 \one_{A_n^s} \, \Bigg| \, \Fcal_{s-1} \right].
\end{align*}

Note that because $\Qcal$ is countable, both $R_{\infty,n}(\delta, \eta)$ and $R_{2,n}(\delta, \eta)$ are well defined and random.
So that on the event $A_n^n$,
\begin{align*}
|C_t^{\ell}|
 &\leq 2^{\ell-1} R_{\infty,n}(\delta, \eta)^{\ell-2} R_{2,n}(\delta, \eta)^2,
\end{align*}
and since $2^{\ell-1} \leq \ell!$ for all $\ell \geq 2$, on the event $A_n^n$,
\begin{equation}
\label{eq_majoration_C}
|C_{n+1}^{\ell}|
 \leq \frac{\ell!}{2} R_{2,n}(\delta, \eta)^2 R_{\infty,n}(\delta, \eta)^{\ell-2}.
\end{equation}
By definition, on the event $A_n^n$,
\begin{equation}
\label{eq_prop_R2}
\left\{
\begin{aligned}
\forall \delta, \eta \in \Theta_{m'} \cup \{\star\}, \quad
 & R_{2,n}(\delta, \eta) \leq F^\infty_{m'} \sqrt{2n}, \\
\forall \delta, \eta \in \Theta_{m'}, \quad
 & R_{2,n}(\delta, \eta) \leq F^{\text{lip}}_{m'} \|\delta - \eta\|_{m'} \sqrt{2n}.
\end{aligned}
\right.
\end{equation}
Finally, since $R_{2,n}(\delta,\eta)$ is the Euclidean norm of the vector whose coordinate $s \in [n]$ is the $\Lbf^2(\Xcal, p^\star(\cdot | 
\Fcal_{s-1}) d\dommes)$ distance between $\log p^{m'}_{\eta,s}(\cdot | 
\Fcal_{s-1}) \one_{A_n^s}$ and $\log p^{m'}_{\delta,s}(\cdot | 
\Fcal_{s-1}) \one_{A_n^s}$, it satisfies the triangular inequality: for all $\eta,\delta,\theta \in \Theta_{m'} \cup \{\star\}$,
\begin{equation}
\label{eq_IT_R2}
R_{2,n}(\delta, \eta)
 \leq R_{2,n}(\delta, \theta) + R_{2,n}(\theta, \eta).
\end{equation}

Likewise, on the event $A_n^n$,
\begin{equation}
\label{eq_prop_Rinfty}
\left\{
\begin{aligned}
\forall \delta, \eta \in \Theta_{m'} \cup \{\star\}, \quad
 & R_{\infty,n}(\delta, \eta) \leq F^\infty_{m'}, \\
\forall \delta, \eta \in \Theta_{m'}, \quad
 & R_{\infty,n}(\delta, \eta) \leq F^{\text{lip}}_{m'} \|\delta - \eta\|_{m'}.
\end{aligned}
\right.
\end{equation}
and for all $\eta,\delta,\theta \in \Theta_{m'} \cup \{\star\}$
\begin{equation}
\label{eq_IT_Rinfty}
R_{\infty,n}(\delta, \eta)
 \leq R_{\infty,n}(\delta, \theta) + R_{\infty,n}(\theta, \eta).
\end{equation}

Identify $\Theta_{m'} \cup \{\star\}$ with the subset $\widetilde{\Theta}_{m'}$ of the vector space $\Rbb^{D_{m'}+1}$ of generic element $(\delta,u)$ with $\delta \in \Rbb^{D_{m'}}$ and $u \in \Rbb$, defined as
\begin{equation*}
\widetilde{\Theta}_{m'}
 = \{ (\theta,0) : \theta \in \Theta_{m'} \}
 \cup \{ (\bar{\theta}^{m'},1) \},
\end{equation*}
for some fixed $\bar{\theta}^{m'} \in \Theta_{m'}$. Endow the vector space $\Rbb^{D_{m'}+1}$ with the norms
\begin{equation*}
\widetilde{N}_\infty((\delta,u))
 = F^{\text{lip}}_{m'} \|\delta\|_{m'}
 + F^\infty_{m'} |u|
\quad \text{and} \quad
\widetilde{N}_2((\delta,u))
 = \sqrt{2n} \widetilde{N}_\infty((\delta,u)).
\end{equation*}

Under Assumption~\ref{hyp_lipschitz}, for any $\delta, \eta \in \Theta_{m'}$, by~\eqref{eq_prop_R2} and~\eqref{eq_prop_Rinfty}, on the event $A_n^n$,
\begin{equation*}
\left\{
\begin{aligned}
R_{2,n}(\delta,\eta)
 \leq F^{\text{lip}}_{m'} \|\delta - \eta\|_{m'} \sqrt{2n}
 = \widetilde{N}_2((\delta,0) - (\eta,0)) \\
R_{\infty,n}(\delta,\eta)
 \leq F^{\text{lip}}_{m'} \|\delta - \eta\|_{m'}
 = \widetilde{N}_\infty((\delta,0) - (\eta,0)),
\end{aligned}
\right.
\end{equation*}
and both inequalities extend to $\eta = \star$ since, for $R_{2,n}$,
\begin{align*}
R_{2,n}(\delta,\star)
 &\leq F^\infty_{m'} \sqrt{2n} \\
 &\leq F^{\text{lip}}_{m'} \|\delta - \bar{\theta}^{m'}\|_{m'} \sqrt{2n} + F^\infty_{m'} \sqrt{2n} \\
 &= \widetilde{N}_2((\delta,0) - (\bar{\theta}^{m'},1))
\end{align*}
and likewise for $\delta = \star$ and for $R_{\infty,n}$.
Let 
\begin{equation*}
A_{m'} = F^{\text{lip}}_{m'} M_{m'} + F^\infty_{m'}
 \quad \text{and} \quad
v_{m'} = A_{m'} \sqrt{2n},
\end{equation*}
so that $\widetilde{N}_2(\delta-\eta) \leq v_{m'}$ and $\widetilde{N}_\infty(\delta-\eta) \leq A_{m'}$ for all $\delta,\eta \in \widetilde{\Theta}_{m'}$.

We may now apply Theorem 5 of \cite{aubert:hal-04526484} to the process $Y_\delta = n \nu(g_{\delta}^{m'})$ indexed by $\delta \in \Theta_{m'} \cup \{\star\}$, with $c=0$ and the event $A = A_n^n$: for all $\sigma> 0$ and $x \geq 0$, let
\begin{equation*}
 \Psi(\sigma,x) =
 20 \left[ (\sigma \wedge v_{m'}) \sqrt{x + (D_{m'}+1) \log \left(\frac{v_{m'}}{\sigma} \vee e \right)}
 + A_{m'} \left(x + (D_{m'}+1) \log \left(\frac{v_{m'}}{\sigma} \vee e \right) \right) \right].
\end{equation*}
Then, for all $\theta \in \Theta_{m'} \cup \{\star\}$, $\sigma > 0$ and $x \geq 0$,
\begin{multline*}
\Pbb \left( \left\{ \sup_{\delta \in \Theta_{m'} \cup \{\star\}} \frac{Y_{\delta}-Y_{\theta}}{R_{2,n}(\delta,\theta)^2 + \sigma^2} \geq 4\sigma^{-2} \Psi(\sigma,x + D_{m'}) \right\} \cap A_n^n \right) \\
 \leq \left(2 \log \left(\frac{v_{m'}}{\sigma}\right) \vee 0 + 1\right) e^{-(x+D_{m'})}.
\end{multline*}

In particular, by taking the union bound over $m' \in \Mcal$ for $\theta = \star$, for any family of positive numbers $(y_{m'})_{m' \in \Mcal}$, with probability at least $\Pbb(A_n^n) - \displaystyle e^{-x} \sum_{m' \in \Mcal} \left(2\log \left(\frac{v_{m'}}{y_{m'}}\right) \vee 0 + 1\right) e^{-D_{m'}}$, for all $m' \in \Mcal$,
\begin{multline*}
\sup_{\delta\in \Theta_{m'}} \left(\frac{n \nu(g_{\delta}^{m'})}{R_{2,n}(\delta,\star)^2 + y_{m'}^2}\right)
 \\
 \leq \frac{80}{y_{m'}^2} \Bigg(
 y_{m'} \sqrt{x} + (y_{m'} \wedge v_{m'}) \sqrt{(D_{m'}+1) \log \left(\frac{v_{m'}}{y_{m'}} \vee e \right)}
 \\
 + A_{m'} x + A_{m'} (D_{m'}+1) \log \left(\frac{v_{m'}}{y_{m'}} \vee e \right)
 \Bigg).
\end{multline*}

For each $m' \in \Mcal$, let $\sigma_{m'}$ be the solution of the equation
\begin{equation*}
\sigma
 = \left(1 \wedge \frac{v_{m'}}{\sigma} \right) \sqrt{(D_{m'}+1) \log \left(\frac{v_{m'}}{\sigma} \vee e \right)}
 + \frac{A_{m'}}{\sigma} (D_{m'}+1) \log \left(\frac{v_{m'}}{\sigma} \vee e \right),
\end{equation*}
which exists since the right hand side is positive and non-increasing on $(0,+\infty)$. For any family $(y_{m'})_{m' \in \Mcal}$ such that $y_{m'} \geq \sigma_{m'}$ for all $m' \in \Mcal$, for any $x \geq 0$, it holds with probability at least $\Pbb(A_n^n) - \displaystyle e^{-x} \sum_{m' \in \Mcal} \left(2\log \left(\frac{v_{m'}}{y_{m'}}\right)\vee 0 + 1\right) e^{-D_{m'}}$ that for all $m' \in \Mcal$,
\begin{align*}
\sup_{\delta\in \Theta_{m'}} & \left(\frac{n \nu(g_{\delta}^{m'})}{R_{2,n}(\delta,\star)^2 + y_{m'}^2}\right)
 \leq \frac{80}{ y_{m'}} \left(\sigma_{m'} + \sqrt{x+D_{m'}} + \frac{A_{m'}}{y_{m'}}(x+D_{m'}) \right).
\end{align*}
Let $\eta \in (0,1)$ and fix for each $m' \in \Mcal$
\begin{equation*}
 y_{m'} = \eta^{-1} \sqrt{\sigma_{m'}^2 + x+D_{m'}} \ ,
\end{equation*}
for all $x \geq 0$, with probability at least $\Pbb(A_n^n) - \displaystyle e^{-x} \sum_{m' \in \Mcal} \left(2\log \left(\frac{v_{m'}}{y_{m'}}\right)\vee 0 + 1\right) e^{-D_{m'}}$, for all $m' \in \Mcal$,
\begin{equation*}
\sup_{\delta \in \Theta_{m'}} \left(\frac{\nu(g_{\delta}^{m'})}{\frac{1}{n} R_{2,n}(\delta,\star)^2 + \frac{1}{n} y_{m'}^2}\right)
 \leq 80 (2 \eta + \eta^2 A_{m'}),
\end{equation*}
where 
\begin{itemize}
 \item the first term on the right hand side is due to the concavity of $x \in (0,+\infty) \mapsto \sqrt{x}$,
 \item the second term on the right hand side holds because $x+D_{m'} \leq x + D_{m'}+\sigma_{m'}^2 = \eta^2 y_{m'}^2$.
\end{itemize}
By definition $y_{m'} \geq \eta^{-1} \sqrt{D_{m'}} \geq 1$, and $v_{m'} = A_{m'} \sqrt{2n} \geq e$. Therefore,
\begin{equation*}
2 \log \left(\frac{v_{m'}}{y_{m'}}\right) \vee 0 + 1
 \leq 3 \log v_{m'}.
\end{equation*}
The control of $-\nu(g_\delta^{m'})$ is identical, hence we may control all $|\nu(g_\delta^{m'})|$ with probability at least $\Pbb(A_n^n) - \displaystyle 6 e^{-x} \sum_{m' \in \Mcal} \log(v_{m'}) e^{-D_{m'}}$ by union bound. To conclude the proof of the Lemma, note that $\frac{1}{n} R_{2,n}(\delta,\star)^2 \leq 2\Vbf_n(p_\delta)$.

\subsection{Proof of Lemma~\ref{lemmaVariance}}
\label{section_proofoflemmaVariance}

Let's begin with a result proved in \cite{shen2013adaptive} which we slightly adapt to our situation. 

\begin{lemma}[Adaptation of Lemma 4 in \cite{shen2013adaptive}]\label{lemma_4STG}
For any probability measures $P$ and $Q$ with densities $p$ and $q$, and any $\lambda \in (0,1/2]$,
\begin{equation*}
P \left\{ \left(\log\frac{p}{q}\right)^2\one_{\left| \log\left(\frac{p}{q}\right)\right|\leq \log\left(\frac{1}{\lambda}\right)} \right\}
 \leq 8 \left(1 + \left(\log\frac{1}{\lambda}\right)^2 \right) P \left\{ \left(\frac{q^{1/2}}{p^{1/2}}-1\right)^2\one_{\left| \log\left(\frac{p}{q}\right)\right|\leq \log\left(\frac{1}{\lambda}\right)} \right\}.
\end{equation*}
\end{lemma}

\begin{proof}
 See Section~\ref{section_proofoflemma4STG}.
\end{proof}

Let us apply Lemma \ref{lemma_4STG} to $p = p^\star_s$ and $q = p_{\delta,s}^{m'}$ with $\lambda = \exp(-F^\infty_{m'})$ for all $s \in [n]$ and $\delta \in \Theta_{m'}$. By definition of $F^\infty_{m'}$, $\lambda \leq 1/2$, so that for all $s \in [n]$ and $\delta \in \Theta_{m'}$,
\begin{multline*}
\int \left| \log \frac{p_s^\star(x_s | \Fcal_{s-1})}{p_{\delta,s}^{m'}(x_s | \Fcal_{s-1})} \right|^2 \one_{\left|\log \frac{p_s^\star(x_s | \Fcal_{s-1})}{p_{\delta,s}^{m'}(x_s | \Fcal_{s-1})}\right| \leq F^\infty_{m'}} p^\star_s(x_s | \Fcal_{s-1})d\dommes(x_s) \\
 \leq 16 h^2\left(p^\star_s(\cdot | \Fcal_{s-1}), p_{\delta,s}^{m'}(\cdot | \Fcal_{s-1}) \, \Big| \,\Fcal_{s-1} \right) (1 + (F^\infty_{m'})^2),
\end{multline*}
where
\begin{equation*}
h^2\left(p_t(\cdot | \Fcal_{t-1}), q_t(\cdot | \Fcal_{t-1}) \, \Big| \, \Fcal_{t-1} \right)
 = \frac{1}{2} \int \left( \sqrt{p_t(x_t | \Fcal_{t-1})} - \sqrt{q_t(x_t | \Fcal_{t-1})} \right)^2 d\dommes(x_t).
\end{equation*}

Let us recall a classical relation between the Hellinger distance and the Kullback-Leibler divergence, see for instance in \cite[Lemma 7.23]{massart2007concentration}: for any probability measures $P$ and $Q$,
\begin{equation*}
 2 h^2(P,Q) \leq \DKL(P,Q).
\end{equation*}
Applying this inequality to the probability measures $P = p^\star_t(\cdot | X_1^{t-1})$ and $Q = p_{t,\hat{\theta}^{m'}}^{m'}(\cdot | X_1^{t-1})$ conditionally to $X_1^{t-1}$ for all $t \in [n]$ shows that
\begin{equation*}
\Vbf_n(p_{\delta}^{m'})
 \leq 8(1 + (F^\infty_{m'})^2) \Kbf_n(p_{\delta}^{m'})
 \leq 16 (F^\infty_{m'})^2 \Kbf_n(p_{\delta}^{m'})
\end{equation*}
since $F^\infty_{m'} \geq 1$.

\subsection{Proof of Lemma~\ref{lemma_4STG}}\label{section_proofoflemma4STG}

The proof follows exactly the same steps as \cite{shen2013adaptive}. Let $r : (0,+\infty) \to \Rbb$ be the function defined implicitly by 
\begin{equation*}
 \log(x) = 2(x^{1/2}-1)-r(x)(x^{1/2}-1)^2.
\end{equation*}
The function $r$ is non-negative, decreasing, and $r(x) \leq 2 \log(1/x)$ for all $x \in (0,1/2]$ (see e.g.~\cite{ghosal2007posterior}). Let $\lambda \in (0,1/2]$. Since for any $x \geq 1$, $|\log(x)|\leq 2 |x^{1/2}-1|$,
\begin{equation*}
P \left\{ \left(\log\frac{p}{q}\right)^2\one_{1 \leq \frac{q}{p}\leq \frac{1}{\lambda}} \right\}
 \leq 4 P \left\{ \left(\frac{q^{1/2}}{p^{1/2}}-1\right)^2 \one_{1 \leq \frac{q}{p}\leq \frac{1}{\lambda}} \right\}.
\end{equation*}
Moreover, by definition of $r$,
\begin{align*}
P &\left\{ \left(\log\frac{p}{q}\right)^2\one_{\lambda \leq \frac{q}{p}\leq 1} \right\}\\
 &\leq 8 P \left\{ \left(\frac{q^{1/2}}{p^{1/2}}-1\right)^2\one_{\lambda \leq \frac{q}{p}\leq 1} \right\}
 + 2P \left\{ r^2\left(\frac{q}{p}\right)\left(\frac{q^{1/2}}{p^{1/2}}-1\right)^4\one_{\lambda \leq \frac{q}{p}\leq 1} \right\}\\
 &\leq 8 P \left\{ \left(\frac{q^{1/2}}{p^{1/2}}-1\right)^2\one_{\lambda \leq \frac{q}{p}\leq 1} \right\}+2r^2(\lambda)P \left\{\left(\frac{q^{1/2}}{p^{1/2}}-1\right)^2\one_{\lambda \leq \frac{q}{p}\leq 1} \right\}\\
 &\leq 8 P \left\{ \left(\frac{q^{1/2}}{p^{1/2}}-1\right)^2\one_{\lambda \leq \frac{q}{p}\leq 1} \right\}+8 \left(\log\frac{1}{\lambda}\right)^2 P \left\{\left(\frac{q^{1/2}}{p^{1/2}}-1\right)^2\one_{\lambda \leq \frac{q}{p}\leq 1} \right\},
\end{align*}
where
\begin{itemize}
 \item the first inequality holds because for any $a,b \in \Rbb$, $(a+b)^2 \leq 2a^2+2b^2$,
 \item the second inequality holds because $r$ is decreasing and since $0\leq \frac{q}{p}\leq 1$, $\left(\frac{q^{1/2}}{p^{1/2}}-1\right)^2 \leq 1$,
 \item the third inequality holds because $r(x) \leq 2 \log(1/x)$ for $x \in (0, 1/2]$.
\end{itemize}
All in all,
\begin{align*}
P\left\{ \left(\log\frac{p}{q}\right)^2\one_{\lambda \leq \frac{q}{p}\leq \frac{1}{\lambda}} \right\}
 \leq 8 \left(1+\left(\log\frac{1}{\lambda}\right)^2\right) P \left\{ \left(\frac{q^{1/2}}{p^{1/2}}-1\right)^2\one_{\lambda \leq \frac{q}{p}\leq \frac{1}{\lambda}} \right\}.
\end{align*}

\end{document}